\documentclass[11pt]{amsart}
\usepackage{amsbsy,amssymb,amscd,amsfonts,latexsym,amstext,delarray, amsmath,graphicx,color,caption, amsthm, enumerate}
\usepackage{graphicx}
\usepackage{mathtools}
\usepackage{hyperref}
\usepackage{epsfig}
\input xy

\newcommand\Z{{\mathbb Z}}
\newcommand\R{{\mathbb R}}
\newcommand\C{{\mathbb C}}

\newcommand\grad{{ \nabla }}

\hypersetup{
	colorlinks=true,
	linkcolor=blue,
	filecolor=blue,      
	urlcolor=blue,
	citecolor = blue
}

\begin{document}
	
	\newtheorem{example}{Example}[section]
	\newtheorem{lemma}{Lemma}[section]
	\newtheorem{thm}{Theorem}
	\newtheorem{prop}[lemma]{Proposition}
	\newtheorem{cor}{Corollary}[section]
	
	\title[Weighted Estimates for Rough Bilinear Operators]{Weighted Estimates for Rough Bilinear Singular Integrals via Sparse Domination}
	\author{Alexander Barron}
	\maketitle
	
	\newcommand{\Addresses}{{
			\bigskip
			\footnotesize

			\textsc{Department of Mathematics, Brown University,
				Providence, RI 02906, USA}\par\nopagebreak
			\textit{E-mail address}: \texttt{alexander\_barron@brown.edu}

		}}

		\begin{abstract}
			We prove weighted estimates for rough bilinear singular integral operators with kernel $$K(y_1, y_2) = \frac{\Omega((y_1,y_2)/|(y_1,y_2)|)}{|(y_1, y_2)|^{2d}},$$ where $y_i \in \R^{d}$ and $\Omega \in L^{\infty}(S^{2d-1})$ with $\int_{S^{2d-1}}\Omega d\sigma = 0.$ The argument is by sparse domination of rough bilinear operators, via an abstract theorem that is a multilinear generalization of recent work by Conde-Alonso, Culiuc, Di Plinio and Ou \cite{roughSparse}. We also use recent results due to Grafakos, He, and Honz\'{\i}k \cite{BiRough} for the application to rough bilinear operators. In particular, since the weighted estimates are proved via sparse domination, we obtain some quantitative estimates in terms of the $A_{p}$ characteristics of the weights. The abstract theorem is also shown to apply to multilinear Calder\'{o}n-Zygmund operators with a standard smoothness assumption. Due to the generality of the sparse domination theorem, future applications not considered in this paper are expected.  
		\end{abstract}
		\section{Introduction}

		The study of rough singular integral operators dates back to Calder\'{o}n and Zygmund's classic papers \cite{CZ} and \cite{CZ2}. In \cite{CZ2} the authors proved that if $\Omega \in L\log L(S^{n-1})$ with $\int_{S^{n-1}} \Omega \ d\sigma = 0$, then the operator $$R_{\Omega}f(x) := \text{p.v. } \int_{\R^{n}} \frac{\Omega(y/|y|)}{|y|^{n}} f(x-y) dy$$ is bounded on $L^{p}(\R^{n})$ for $1 < p < \infty$. Hoffman \cite{Hof} and  Christ and Rubio de Francia \cite{Christ}, \cite{ChristRubio} proved such operators are bounded from $L^{1} \rightarrow L^{1,\infty}$ for small dimensions, and weak-$(1,1)$ boundedness in arbitrary dimensions was later proved by Seeger \cite{Seeger} and Tao \cite{Tao} (in a more general setting). The weighted theory for such operators was developed during this time as well; for some examples, see the work by Duoandikoetxea \cite{Duo1} and Watson \cite{Watson}. More recently, Conde-Alonso, Culiuc, Di Plinio, and Ou \cite{roughSparse} have shown that the bilinear form associated to a rough operator $R_{\Omega}$ can be bounded by \textit{positive sparse forms} (defined below), proving quantitative weighted estimates for $R_{\Omega}$ as an easy corollary. Also see the recent work by Hyt\"{o}nen, Roncal, and Tapiola \cite{HRT}, where the authors establish quantitative estimates using a different method. 
		
		We are interested in the bilinear analogues of the operators $R_{\Omega}$. The study of these operators originates with work by Coifman and Meyer \cite{CoifMey}. Suppose $\Omega\in L^{q}(S^{2d-1})$ for some $q > 1$ with $\int_{S^{2d-1}}\Omega \ d\sigma = 0$, and define the rough bilinear operator \begin{equation}\label{roughBiDef}T_{\Omega}(f_{1}, f_{2})(x) = \textrm{p.v. } \int_{\R^{d}}\int_{\R^{d}} f_{1}(x - y_{1})f_{2}(x - y_{2})\frac{\Omega((y_1, y_2)/|(y_1, y_2)|)}{|(y_1, y_2)|^{2d}} \ dy_1 dy_2.\end{equation} Grafakos, He, and Honz\'{\i}k \cite{BiRough} have proved using a wavelet decomposition that if $\Omega \in L^{\infty}(S^{2d-1})$ then $$\|T_{\Omega}\|_{L^{p_1}(\R^{d}) \times L^{p_2}(\R^{d}) \rightarrow L^{p}(\R^{d})} < \infty$$ when $1< p_1, p_2 < \infty$ and $\frac{1}{p_1} + \frac{1}{p_2} = \frac{1}{p}.$ Also see \cite{BiRough} for references to earlier work. In this paper we develop a weighted theory for the rough bilinear operators $T_{\Omega}$, using a multilinear generalization of the sparse domination theory from \cite{roughSparse} along with the results by Grafakos, He, and Honz\'{\i}k. Below we briefly review some elements of the sparse domination theory, and then state the main results of the paper.
		
		Recall that a collection of cubes $\mathcal{P}$ is said to be $\eta$-\textit{sparse} if for each $Q\in \mathcal{P}$ there is $E_Q \subset Q$ such that $|E_Q| \geq \eta|Q|$, and such that $E_Q\cap E_{Q'} = \emptyset$ when $Q \neq Q'$ (here $0 < \eta < 1$). We will work with \textit{positive sparse forms}. Let $\mathcal{S}$ be a sparse collection of cubes in $\R^{d}$ and suppose $\vec{p} = (p_1, p_2,...,p_{m+1})$ is an $(m+1)$-tuple of exponents. We define the form $$\textrm{PSF}_{\mathcal{S}}^{\vec{p}}(f_{1},...,f_{m+1}) := \sum_{Q\in \mathcal{S}} |Q|\prod_{i=1}^{m+1}(f_i)_{p_{i}, Q}, $$ where $$(f)_{q, Q} = |Q|^{-1/q}\|f1_{Q}\|_{q}.$$ These operators were initially studied in \cite{BFP}, \cite{MultiSparse}, \cite{DyShift}, motivated by earlier pointwise estimates in the sparse domination theory (see \cite{Lacey}, \cite{Lerner1}, \cite{Lerner2}, and \cite{DyCalc} for some examples). Techniques involving positive sparse forms have lead to several interesting results, and can overcome some technical obstacles appearing in earlier work. For example, one can avoid the reliance on maximal truncation estimates present in \cite{Lacey}, \cite{Lerner2} by using sparse form methods. See \cite{roughSparse}, \cite{KrauseLacey} \cite{LaceyArias}, and \cite{LaceySpencer} for examples of recent developments. A common theme in the theory, present in this paper as well, is that sparse bounds yield quantitative weighted estimates for operators as straightforward corollaries. 
		
		Our main result is the following theorem. 
		
		\begin{thm}\label{MainThm} Suppose $T_{\Omega}$ is the rough bilinear singular integral operator defined in \eqref{roughBiDef}, with $\Omega \in L^{\infty}(S^{2d-1})$ and $\int_{S^{2d-1}} \Omega = 0$. Then for any $1 < p < \infty$, there is a constant $C_p > 0$ so that $$|\langle T_{\Omega}(f_1,f_2), f_3 \rangle|  \leq C_{p}\|\Omega\|_{L^{\infty}(S^{2d-1})}\sup_{\mathcal{S}} \textrm{PSF}_{\mathcal{S}}^{(p,p,p)}(f_{1},f_2, f_3).$$ 
		\end{thm}
		
		\noindent Here the supremum is taken over all sparse collections $\mathcal{S}$ with some fixed sparsity constant $\eta$ that does not depend on the functions. This theorem is a consequence of a more general multilinear sparse domination result, which is stated in Section 2. As an application of Theorem \ref{MainThm} we derive weighted estimates for $T_{\Omega}$. Recall the $A_{p}$ class of weights $w$, where $w \in A_{p}$ for $1 < p < \infty$ if $w > 0$, $w \in L_{\text{loc}}^{1},$ and $$[w]_{A_{p}} := \sup_{\substack{Q\subset \R^{d} \\ \text{cubes} }} \left(\frac{1}{|Q|}\int_{Q}w \ dx \right) \left( \frac{1}{|Q|}\int_{Q} w^{-\frac{1}{p-1}} \ dx \right)^{p-1} < \infty.$$ Below we write $L^{p}(w)$ for the space $L^{p}$ with measure $w(x)dx$.

		\begin{cor}\label{MultiWeight} Fix $1 < p_1, p_2 < \infty$ and $\frac{1}{2} < p < \infty$ such that $\frac{1}{p_1} + \frac{1}{p_2} = \frac{1}{p}$. Then for all weights $(w_1^{p_1}, w_2^{p_2})$ in $(A_{p_1}, A_{p_2})$, there is a constant $C$ depending on $[w^{p_1}]_{A_{p_1}}, [w^{p_2}]_{A_{p_2}}, d, p_1$ and $ p_2$ such that $$\|T_{\Omega}(f_1, f_2)\|_{L^{p}(w_1^{p}w_2^{p})} \leq C\|f_1\|_{L^{p_1}(w_1^{p_1})}\|f_2\|_{L^{p_2}(w_2^{p_2})}$$ for all $f_i \in L^{p_i}(w_i^{p_i})$. \end{cor}
		
		\noindent These estimates were originally proved by Cruz-Uribe and Naibo in \cite{CruzNaibo} with a different technique. Our proof uses the sparse domination from Theorem \ref{MainThm} along with methods from \cite{MultiSparse}, \cite{DyCalc}, \cite{LernerMulti} and extrapolation. Note the following special case of Corollary \ref{MultiWeight} in the single-weight case. 
		
		\begin{cor}\label{WeightedEst} Suppose $1 < q < \infty$ and $\Omega \in L^{\infty}(S^{2d-1})$ with $\int_{S^{2d-1}} \Omega \ d\sigma = 0$. Then if $w \in A_{q}$, there is a constant $C = C(w,q,\Omega)$ such that $$\|T_{\Omega}(f,g)\|_{L^{q/2}(w)} \leq C\|f\|_{L^{q}(w)} \|g\|_{L^{q}(w)} $$ for all $f,g \in L^{p}(w)$. \end{cor}

		\noindent We provide a separate proof of this corollary in Section 5.1 that indicates how to track the dependence of $C$ on $[w]_{A_p}$. The proof of Corollary \ref{WeightedEst} is again a consequence of sparse domination and extrapolation.
		
		We can also prove weighted estimates with respect to the more general \textit{multilinear} Muckenhoupt classes. In particular, suppose  $\sum_{i=1}^{3}\frac{1}{q_i} = 1$ with $1 < q_1,q_2,q_3 < \infty$ and let  $v_1,v_2,v_3$ be strictly positive functions such that $$\prod_{i=1}^{3} v_{i}^{\frac{1}{q_i}} = 1.$$ Define $$[\vec{v}]_{A_{\vec{q}}^{\vec{p}}} := \sup_{Q}\prod_{i=1}^{3} \left(\frac{1}{|Q|}\int_{Q}v_{i}^{ \frac{p_i}{p_i - q_i} }\right)^{\frac{1}{p_i} - \frac{1}{q_i}}$$ for any tuple $\vec{p} = (p_1, p_2, p_3)$ with $1 \leq p_i < q_i < \infty$ for $i = 1,2,3.$ Notice that we assume $v_i > 0$ for each $i$, but we do not assume that $v_i \in L_{loc}^{1}$. This multilinear class was originally introduced in \cite{LernerMulti} with $\vec{p} = (1,1,1)$, and used in \cite{MultiSparse} for more general $\vec{p}$. In the case when $\vec{p} = (1,1,1)$, it is the natural weight class associated to the maximal operator $$\mathcal{M}(\vec{f})(x) := \sup_{x \in Q} \prod_{i=1}^{2}\frac{1}{|Q|}\int_{Q} |f_i(y_i)| dy_i.$$ We will prove the following in Section 5.2.
		
		\begin{cor}\label{MultiWeight2} Suppose the tuples $\vec{p}, \vec{q}$ and the functions $v_1,v_2,v_3$ are defined as above, with $p_i > 1$ for $i = 1,2,3$. Also let $$\sigma = v_3^{-q_3' /q_3} = v_1^{-\frac{q_2}{q_1 + q_2}}v_2^{-\frac{q_1}{q_1 + q_2}},$$ where $q_3'$ is the conjugate of $q_3$. Then if $[\vec{v}]_{A_{\vec{q}}^{\vec{p}}} < \infty$ and $\Omega \in L^{\infty}(S^{2d-1})$ with $\int_{S^{2d-1}}\Omega = 0$, we have $$\|T_{\Omega}(f_1, f_2)\|_{L^{q_3 '}(\sigma)} \leq C_{\Omega,d, p_i, q_i} [\vec{v}]_{A_{\vec{q}}^{\vec{p}}}^{\max \left\{ \frac{q_i}{q_i-p_i} \right\} }\|f_1\|_{L^{q_1}(v_1)}\|f_2\|_{L^{q_2}(v_2)}$$ for all $f_i \in L^{q_i}(v_i), \ i = 1,2$. \end{cor}
		
		\noindent We will see in Section 5.2 below that Corollary \ref{MultiWeight} and Corollary \ref{MultiWeight2} can both be deduced from the same lemma, which is a consequence of the sparse domination. However, the class $[\vec{v}]_{A_{\vec{q}}^{\vec{p}}}$ is in general strictly larger than $A_{q_1} \times A_{q_2}$ since, for example, $v_1$ and $v_2$ do not have to be in $L_{loc}^{1}$ (see \cite{LernerMulti} for some examples when $\vec{p} = (1,1,1)$). 
		
		\vspace{5mm}
		
		\noindent\textbf{Sample Application.} We note that the Calder\'{o}n commutator $$ \mathcal{C}(a,f)(x) = \text{p.v.} \int_{\R} \frac{A(x) - A(y)}{(x-y)^{2}} f(y)dy,$$ where $a$ is the derivative of $A$, is an example of the rough bilinear operators considered in this paper \cite{CoifMey}. Let $e(t) = 1$ if $t >0$ and $e(t) = 0$ if $t<0$. Then we can write this operator as $$\text{p.v.} \int_{\R}\int_{\R} K(x-y, x-z) f(y)a(z) dydz$$ with $K(y,z) = \frac{e(z) - e(z-y)}{y^{2}} =: \frac{\Omega((y,z)/|(y,z)|)}{|(y,z)|^{2}}$. Moreover, $\Omega$ is odd and bounded, so Theorem \ref{MainThm} and the weighted corollaries all apply to $\mathcal{C}(a,f)$.

		\subsection{Structure of Paper} The proof of Theorem \ref{MainThm} is broken up into two parts. We begin by formulating an abstract sparse domination theorem in the multilinear setting (Theorem \ref{abstractThm} below), a result that generalizes Theorem C in \cite{roughSparse} by Conde-Alonso, Culiuc, Di Plinio, and Ou. The proof of this theorem is similar to their result, so it is deferred to Section 6. In the second part of the paper we use the deep results of Grafakos, He, and Honz\'{\i}k to show that the assumptions of Theorem \ref{abstractThm} are satisfied by the rough bilinear operators $T_{\Omega}$. Along the way we also prove a sparse domination result for multilinear Calder\'{o}n-Zygmund operators with a standard smoothness assumption, as defined by Grafakos and Torres in \cite{GrTo}. This is the content of Theorem \ref{CZsparse} below. This sparse domination result for multilinear Calder\'{o}n-Zygmund operators uses different methods than the recent paper by K. Li \cite{KLi}, and in particular avoids reliance on maximal truncations. As a corollary we can partially recover the weighted estimates from \cite{DyCalc}. Finally, in Section 5 we prove the weighted estimates outlined in the corollaries above. 
		
		Theorem \ref{abstractThm} is independently interesting, and due to its generality we expect it to be useful for analyzing other multilinear operators in the future. We also note that Theorem \ref{MainThm} only applies when $\Omega \in L^{\infty}(S^{2d-1})$, since we need some analogue of the classical Calder\'{o}n-Zygmund size condition. Sparse domination when $\Omega \in L^{q}(S^{2d-1})$ for other values of $q$, or when $\Omega$ is in some Orlicz-Lorentz space, is still an open problem. In fact, the full range of boundedness of $T_{\Omega}$ when $\Omega\in L^{q}(S^{2d-1})$ for $2 < q <\infty$ is not known. It is also unknown whether $T_{\Omega}$ is bounded anywhere for $q <2$. See \cite{BiRough} for more details. 
		
		\subsection{Acknowledgments} The author would like to thank Francesco Di Plinio for suggesting the problem that led to this paper, and for reading early drafts and making many helpful suggestions along the way. Section 5 in particular benefited from these suggestions. The author would also like to thank Jill Pipher for many helpful conversations.

		\subsection{Notation and Definitions} Given a dyadic cube $L$, we let $s_L = \log_2(\text{length}(L))$ and let $\hat{L}$ be the $2^{5}$-fold dilate of $L$. Throughout the paper we fix a dyadic lattice $\mathcal{D}$ in $\R^{d}$. A collection of disjoint cubes $\mathcal{P} \subset \mathcal{D}$ will be called a \textit{stopping collection} with top $Q$ if its elements are contained in $3Q$ and satisfy the following \textit{separation properties}: 
		\begin{enumerate}[(i)]
			\item If $L,R\in \mathcal{P}$ and $|s_L - s_R| \geq 8$ then $7L\cap 7R = \emptyset$
			\item $\bigcup_{\substack{L \in \mathcal{P} \\ 3L \cap 2Q \neq \emptyset}} 9L \subset \bigcup_{L\in\mathcal{P}} L$.
		\end{enumerate} This definition is taken from \cite{roughSparse} (the particular constants here are chosen for technical reasons related to the proof of Theorem \ref{abstractThm} below). 
		
		Throughout the paper we use $c_\alpha$ to represent a positive constant depending on the parameter $\alpha$ that may change line to line. We often write $A \lesssim B$ to mean $A \leq cB$, where $c$ is a positive constant depending on the dimension $d$, the multilinearity term $m$, or relevant exponents. For $E \subset \R^{d}$ we let $|E|$ denote its Lebesgue measure and $\textbf{1}_E$ its indicator function. Finally, we will use $M_p(f)(x)$ to denote the $p$-th Hardy-Littlewood maximal function $$M_{p}(f)(x) = \sup_{x\in Q} \left(\frac{1}{|Q|}\int_{Q} |f(y)|^{p} dy \right)^{1/p}$$ (here the supremum is taken over cubes $Q\subset \R^{d}$ containing $x$). Recall that $M_p$ is bounded on $L^{r}(\R^{d})$ when $r > p$. We also write $M^{w}_p(f)$ to denote the $p$-th maximal function associated to a weight $w$. This operator satisfies the same boundedness properties as the standard maximal operator when $w$ is doubling (and in particular when $w \in A_q$ for some $q$).

		\section{Abstract Sparse Theorem in Multilinear Setting}
		
		In this section we formulate an abstract sparse domination result which we will apply in Section 4 to prove Theorem 1. Let $T$ be a bounded $m$-linear operator mapping $L^{r_1}\times ... \times L^{r_m} \rightarrow L^{\alpha}$ for some $r_{i}, \alpha \geq 1$ with $\frac{1}{r_1} + ... + \frac{1}{r_m} = \frac{1}{\alpha},$ and assume $T$ is given by integration against a kernel $K(x_1,...,x_{m+1})$ away from the diagonal. We will assume the kernel of $T$ has a decomposition \begin{equation}\label{ss1} K(x_1,...,x_m, x_{m+1}) = \sum_{s \in \Z} K_{s}(x_1,...,x_m, x_{m+1}),\end{equation} such that in the support of $K_{s}$ we have $|x_k - x_l| \leq 2^{s}$ for all $l,k$. We also define \begin{align*}[K]_{p} := \sup_{s \in \Z} 2^{\frac{mds}{p'}} \sup_{y \in \R^{d}}&\big(\|K_{s}(y, \cdot + y, ... \ , \cdot + y)\|_{L^{p}(\R^{md})} + \\ & \  \|K_{s}(\cdot + y, y, \cdot + y, ... \ , \cdot + y)\|_{L^{p}(\R^{md})} + ...\big)\end{align*}  (the last `...' indicates the other possible symmetric terms), and require that \begin{equation}\label{ss2} [K]_{p} < \infty.\end{equation} This is an abstract analogue of the basic size estimate for multilinear Calder\'{o}n-Zygmund kernels, see \cite{GrTo} and Section 3 below.  We will refer to conditions \eqref{ss1} and \eqref{ss2} as the \textbf{(S)} (single-scale) properties.
		
		Since the operators under consideration are $m$-linear, we will have to deal with $(m+1)$-linear forms of type $$\int_{\R^{n}}T(f_1,...,f_{m})(x) f_{m+1} (x_{m+1}) \ dx_{m+1}.$$ We define $\Lambda^{\nu}_{\mu}(f_1,...,f_{m+1})$ to be the form $$ \int_{\R^{(m+1)d}} \sum_{\mu < s < \nu}  K_{s}(x_1,...,x_{m+1})f_1(x_1)...f_{m+1}(x_{m+1}) \ dx_1...dx_{m+1}, $$ and always assume $\mu > 0$ and $\nu < \infty$. Finally, we will assume the following uniform estimate on the truncations: given $\frac{1}{r_1} + ... + \frac{1}{r_m} = \frac{1}{\alpha}$ as above, $$C_{T}(r_1,...,r_m, \alpha) := \sup \| \Lambda_{\mu}^{\nu}\|_{L^{r_1}\times ... \times L^{r_m} \times L^{\alpha'} \rightarrow \C } < \infty.$$ Here the supremum is taken over all finite truncations. From now on the truncation bounds $\mu, \nu$ will be omitted from the notation unless explicitly needed, and we will write $C_{T}$ in place of $C_{T}(r_1,...,r_{m}, \alpha).$
		
		\subsection{Remark on Truncations} We assume below that our operator is truncated to finitely many scales, and prove estimates that are uniform in the number of truncations. The justification for this assumption is sketched below. The argument is a straightforward generalization of a result that can be found in \cite{Stein} Ch. 1, section 7.2. 
		
		Let $T^{\epsilon}$ denote our operator truncated at scales larger than $\epsilon$ and smaller than $1/\epsilon$. Fix $f_i \in L^{r_i}$ with compact support, $1 \leq i \leq m$. By the uniform bound assumption on $C_{T}$, after passing to a subsequence we can assume $T^{\epsilon}(f_1,...,f_m)$ converges weakly in $L^{\alpha}$ to some $T_{0}(f_1,...,f_{m})$. It is clear from the weak convergence that $T_0$ must be multilinear. We claim that if $Q_i$ is a cube in $\R^{n}$ then \begin{align}\label{truncLim}(T-T_0)&(f_1\textbf{1}_{Q_1}, ..., f_m \textbf{1}_{Q_m})(x_{m+1}) \\ \nonumber &= \textbf{1}_{Q_1}(x_{m+1}) ... \textbf{1}_{Q_m}(x_{m+1})(T- T_0)(f_1,...,f_m)(x_{m+1}) \ \ \text{a.e.}\end{align} In fact, the support restriction on the kernel of $T - T^{\epsilon}$ shows that the integral vanishes if $x_{m+1} \notin Q_{i}$ for any $i$, provided $\epsilon$ is small enough. Then the identity follows from the weak convergence of $T^{\epsilon}$. Using multilinearity we can extend \eqref{truncLim} to simple functions, and then using the boundedness of $T$ we see that \eqref{truncLim} holds with $g_{i} \in L^{r_i}$ in place of $\textbf{1}_{Q_{i}}$. In particular, let $E^{j}$ be an increasing sequence of open sets that exhaust $\R^{d}$, and suppose $f_i \in L^{r_i}$ with support in $E^{j_i}$. Then we must have \begin{align*}(T- T_0)&(f_1,...,f_m)(x_{m+1}) \\ &= f_1 (x_{m+1})...f_m(x_{m+1})(T-T_0)(\textbf{1}_{E^{j_1}},...,\textbf{1}_{E^{j_m}})(x_{m+1}) \ \ a.e.\end{align*} It follows that $T$ differs from the limit $T_0$ by a multiplication operator, provided $$(T-T_0)(\textbf{1}_{E^{j_1}},...,\textbf{1}_{E^{j_m}})(x_{m+1}) =: \phi(x_{m+1})$$ forms a coherent function with respect to the $E^{j_i}$. This is clear as in the linear case. Moreover, $\phi \in L^{\infty}$ since there exists $c_T > 0$ such that \begin{align*} c_T &> \sup_{\substack{f_i\in L^{r_i} \\ \text{cpt. supp } }}\frac{\|(T-T_0)(f_1,...,f_m)\|_{L^{\alpha}}}{\| f_1\|_{L^{r_1}}...\|f_m\|_{L^{r_m}}}  \\ &= \sup_{\substack{f_i\in L^{r_i} \\ \text{cpt. supp } }}\frac{\|\phi\cdot f_1...f_m\|_{L^{\alpha}}}{\|f_1\|_{L^{r_1}}...\|f_m\|_{L^{r_m}}} \geq \|\phi\|_{L^{\infty}}. \end{align*} 
		Therefore, as long as we can prove admissible bounds for multiplication operators of the form $A_{\phi}(f_1,...,f_m) = \phi \cdot f_1 ... f_m$, with $\phi \in L^{\infty}$, we are justified in working with a finite (but otherwise arbitrary) number of scales. 
		
		\subsection{The Abstract Theorem} Assume we are given some stopping collection of cubes $\mathcal{P}$ with top $Q$. We will use the space $\mathcal{Y}_p = \mathcal{Y}_{p}(\mathcal{P})$ from \cite{roughSparse}, with norm $\|\cdot\|_{\mathcal{Y}_p}$ defined by 
		
		\begin{equation}
		\|h\|_{\mathcal{Y}_p} :=\begin{cases} \max \left(\left\|h\textbf{1}_{\R^{d} \backslash \bigcup_{L\in \mathcal{P}}L} \right\|_{\infty}, \  \underset{L \in \mathcal{P}}{\sup} \underset{\ x \in \hat{L}}{\inf} M_{p}h(x) \right) & \text{if $p<\infty$}\\
		\|h\|_{\infty} & \text{if $p = \infty$}
		\end{cases}
		\end{equation}
		
		\noindent (recall from above that $\hat{L}$ is the $2^{5}$-fold dilate of $L$). We let $\|b\|_{\mathcal{X}_{p}}$ denote the $\mathcal{Y}_{p}$-norm of $b$ when $b = \sum_{L\in \mathcal{P}}b_L$ with $b_L$ supported on $L \in \mathcal{P}$, and use $\dot{\mathcal{X}}_p$ to signal that $\int b_L = 0$ for each $L$. Observe that these norms are increasing in $p$, a property we use many times below. Also recall that if $h \in \mathcal{Y}_p$ there is a natural Calder\'{o}n-Zygmund decomposition of $h$ associated to the stopping collection $\mathcal{P}$: we can split $h = g + b$, where \begin{equation} \label{CZdecomp}b = \sum_{L \in \mathcal{P}} \left(h - \frac{1}{|L|}\int_{L} h\right)\textbf{1}_{L}\end{equation} and $$\|g\|_{\mathcal{Y}_{\infty}} \leq 2^{5d}\|h\|_{\mathcal{Y}_p}, \ \ \ \ b\in \dot{\mathcal{X}_p}, \ \|b\|_{\dot{\mathcal{X}_p}} \leq 2^{5d + 1}\|h\|_{\mathcal{Y}_p}.$$ Given a cube $L$, define \begin{align*}\Lambda_{L}(f_1,...,f_m, f_{m+1}) &:= \Lambda^{\min(s_{L}, \text{top trunc})}(f_1 \textbf{1}_{L},f_2,..., f_{m},f_{m+1}) \\ &= \Lambda^{\min(s_{L}, \text{top trunc})}(f_1 \textbf{1}_{L},f_2 \textbf{1}_{3L},...,f_{m}\textbf{1}_{3L}, f_{m+1}\textbf{1}_{3L}), \end{align*} meaning the truncation from above never exceeds the scale of $L$. The second equality follows from the support condition on the kernel imposed by \eqref{ss1}. For simplicity we often write $\Lambda^{s_L}$ to indicate that all truncations are at or below level $s_L$. We will work with \begin{equation}\label{formDef}\Lambda_{\mathcal{P}}(f_1,..., f_{m}, f_{m+1}) := \Lambda_{Q}(f_1,...,f_{m+1}) - \sum_{\substack{L\in \mathcal{P} \\ L \subset Q } } \Lambda_{L}(f_1,..., f_{m+1})\end{equation} (as above, $Q$ is the top cube of the stopping collection $\mathcal{P}$). Note that $\Lambda_{\mathcal{P}}$ is not symmetric in all of its arguments, due to some lack of symmetry in definitions. However, we will see below that $\Lambda_{\mathcal{P}}$ is `almost' symmetric, meaning it is symmetric up to a controllable error term. 
		
		We can now state the abstract sparse domination theorem. 
		
		\begin{thm}\label{abstractThm} Let $T$ be an $m$-linear operator with kernel $K$ as above, such that $K$ can be decomposed as in \eqref{ss1} and $C_T < \infty$. Also let $\Lambda$ be the $(m+1)$-linear form associated to $T$. Assume there exist $1 \leq p_1,...,p_m, p_{m+1} \leq \infty$ and some positive constant $C_L$ such that the following estimates hold uniformly over all finite truncations, all dyadic lattices $\mathcal{D}$, and all stopping collections $\mathcal{P}$: \begin{small}\begin{align}\nonumber |\Lambda_{\mathcal{P}}(b, g_2, g_3,..., g_{m+1})| &\leq C_L|Q|\|b\|_{\dot{\mathcal{X}}_{p_1}}\|g_2\|_{\mathcal{Y}_{p_2}}\|g_3\|_{\mathcal{Y}_{p_3}}...\|g_{m+1}\|_{\mathcal{Y}_{p_{m+1}}} \\  \label{assumptionL} |\Lambda_{\mathcal{P}}(g_1, b, g_3, ..., g_{m+1})| &\leq C_L |Q|\|g_1\|_{\mathcal{Y}_{\infty}}\|b\|_{\dot{\mathcal{X}}_{p_2}}\|g_3\|_{\mathcal{Y}_{p_3}} ... \|g_{m+1}\|_{\mathcal{Y}_{p_{m+1}}} \\ \nonumber  |\Lambda_{\mathcal{P}}(g_1, g_2, b, g_4, ..., g_{m+1})| &\leq C_L |Q|\|g_1\|_{\mathcal{Y}_{\infty}}\|g_2\|_{\mathcal{Y}_{\infty}}\|b\|_{\dot{\mathcal{X}}_{p_3}}\|g_4\|_{\mathcal{Y}_{p_4}}...\|g_{m+1}\|_{\mathcal{Y}_{p_{m+1}}} \\ \nonumber &\vdots \\ \nonumber |\Lambda_{\mathcal{P}}(g_1, g_2, ..., g_{m}, b)| &\leq C_L |Q|\|g_1\|_{\mathcal{Y}_{\infty}}\|g_2\|_{\mathcal{Y}_{\infty}}...\|g_{m}\|_{\mathcal{Y}_{\infty}}\|b\|_{\dot{\mathcal{X}}_{p_{m+1}}}. \end{align} \end{small} Also let $\vec{p} = (p_1,...,p_{m+1})$. Then there is some constant $c_{d}$ depending on the dimension $d$ such that $$\sup_{\mu, \nu} |\Lambda_{\mu}^{\nu}(f_1,...,f_{m}, f_{m+1})| \leq c_{d}\left[C_{T} + C_{L}\right]\sup_{\mathcal{S}} \text{PSF}_{\mathcal{S}; \vec{p}}(f_1,...,f_{m}, f_{m+1})$$ for all $f_{j} \in L^{p_{j}}(\R^{d})$ with compact support, where the supremum is taken with respect to all sparse collections $\mathcal{S}$ with some fixed sparsity constant that depends only on $d,m$. \end{thm}
		
		\subsection{Some Remarks on Theorem \ref{abstractThm}} We will prove in Section 6.1 that the multiplication operators $A_{\phi}$ considered above in Remark 2.1 satisfy admissible PSF bounds. Therefore Theorem \ref{abstractThm} implies $$|\Lambda(f_1,...,f_{m}, f_{m+1})| \leq c_{d}\left[C_{T} + C_{L}\right]\sup_{\mathcal{S}} \text{PSF}_{\mathcal{S}}^{\vec{p}}(f_1,...,f_{m}, f_{m+1})$$ when $f_{j} \in L^{\infty}(\R^{d})$ with compact support. If $\Lambda$ extends boundedly to $L^{q_1}(\R^{d})\times ... \times L^{q_{m+1}}(\R^{d})$, we can use standard density arguments to lift the PSF bound to the case where $f_i \in L^{q_i}(\R^{d})$. 
		
		We also provide some more motivation for the estimates \eqref{assumptionL}. Let $\mathcal{P}$ be a stopping collection of dyadic cubes with top $Q$ and suppose $b = \sum_{L \in \mathcal{P}} b_{L}$ with $b_L$ supported in $L$. Also assume $g_1, g_2, ..., g_{m+1}$ are functions supported in $3Q$. If we fix $L \in \mathcal{P}$ with scale $s_L$, then by definition $\Lambda_{\mathcal{P}}(b_L,g_2, ..., g_{m+1})$ splits as $$\Lambda_{\mathcal{P}}(b_L,g_2, ..., g_{m+1}) = \Lambda^{s_Q}(b_L,g_2, ..., g_{m+1}) - \Lambda^{s_L}(b_L,g_2\textbf{1}_{3L}, ..., g_{m+1}\textbf{1}_{3L}).$$ Now let $s$ be a fixed scale with $s \leq s_L$. Then the piece of $\Lambda_{\mathcal{P}}(b_L,g_{2},..., g_{m+1})$ corresponding to this scale can be written as \begin{align*}
		&\int_{\R^{(m+1)d}}K_{s}b_L (x_1) \left(g_2(x_2) ... g_{m+1}(x_{m+1}) - g_2\textbf{1}_{3L}(x_2) ... g_{m+1}\textbf{1}_{3L}(x_{m+1})\right) d\vec{x} \\ &= \int_{\R^{(m+1)d}}K_{s}b_L  \left(g_2\textbf{1}_{3L} ... g_{m+1}\textbf{1}_{3L} - g_2\textbf{1}_{3L} ... g_{m+1}\textbf{1}_{3L}\right) d\vec{x} \\ &= 0. \end{align*} Here we've used the truncation of the kernel: since $x_1 \in L$ and $|x_1 - x_i| \leq 2^{s} \leq 2^{s_L}$ for any $i$, we must have $x_i \in 3L$ for each $i$. Therefore all scales $s$ entering into $\Lambda_{\mathcal{P}}(b_L,g_{2},..., g_{m+1})$ satisfy $s > s_L$, and as a consequence we have the decomposition $$ \Lambda_{\mathcal{P}}(b_L,g_{2},..., g_{m+1}) = \int_{\R^{(m+1)d}}\sum_{l \geq 1}K_{s_L + l}b_L (x_1)g_2 (x_2)... g_{m+1}(x_{m+1}) d\vec{x}.$$ Summing over $L \in \mathcal{P}$ then gives \begin{align} \label{ssRep1} \Lambda_{\mathcal{P}}(b,g_{2},..., g_{m+1}) = \sum_{L \in \mathcal{P}} \int_{\R^{(m+1)d}}\sum_{l \geq 1}K_{s_L + l}b_L (x_1)g_2 (x_2)... g_{m+1}(x_{m+1}) d\vec{x} \\ \nonumber = \int_{\R^{(m+1)d}}\sum_{s}\sum_{l \geq 1}K_{s}b_{s - l} (x_1)g_2 (x_2)... g_{m+1}(x_{m+1}) d\vec{x},\end{align} where $$b_{s-l} = \sum_{\substack{ L \in \mathcal{P} \\ s_L = s - l }}b_L.$$ The representation \eqref{ssRep1} often enables one to verify the estimates \eqref{assumptionL} in practice. We will see this in the case of multilinear Calder\'{o}n-Zygmund operators below in Section 3. In the Appendix we prove the following result about the adjoint forms:
		
		\begin{prop}\label{adjointDecomp} Let $b, g_1,...,g_{m+1}$ be defined as above, fix $1 < p \leq \infty$, and suppose $[K]_{p'} < \infty$. If we set $$b^{in} = \sum_{\substack{ L \in \mathcal{P} \\ 3L \cap 2Q \neq \emptyset }} b_L,$$ then \begin{align*}\Lambda_{\mathcal{P}}&(g_1, b, g_3, ..., g_{m+1}) \\ &= \int_{\R^{(m+1)d}}\sum_{s}\sum_{l \geq 1}K_{s}b^{in}_{s - l} (x_2)g_1 (x_1)g_3(x_3)... g_{m+1}(x_{m+1}) d\vec{x} \ + \ \phi,\end{align*} where $\phi$ is an error term that satisfies the estimate $$|\phi| \lesssim [K]_{p'}|Q| \|b\|_{\mathcal{X}_{1}} \|g_1\|_{\mathcal{Y}_p}\|g_3\|_{\mathcal{Y}_p}...\|g_{m+1}\|_{\mathcal{Y}_p}.$$ An analogous decomposition holds for each of the other forms $$\Lambda_{\mathcal{P}}(g_1,g_2, b, g_4, ..., g_{m+1}), ... , \Lambda_{\mathcal{P}}(g_1,...,g_{m-1}, b, g_{m+1}), \Lambda_{\mathcal{P}}(g_1,...,g_m , b),$$ with error terms satisfying the same bounds (with obvious changes in indicies). \end{prop}
		
		\noindent In practice the error term $\phi$ does not cause any issues. This quantifies the `almost symmetry' of $\Lambda_{\mathcal{P}}$ mentioned above.  
		
		\vspace{5mm}
		
		The proof of Theorem \ref{abstractThm} is similar to the proof of the abstract Theorem C in \cite{roughSparse}, so we postpone the argument to Section 6 below.    
		
		\section{Multilinear Calder\'{o}n-Zygmund Operators}
		
		In this section we use Theorem \ref{abstractThm} to prove sparse bounds for multilinear Calder\'{o}n-Zygmund (CZ) operators that satisfy a certain smoothness conditon (defined below). For the theory of these operators see \cite{MFA} or \cite{GrTo}. In the process we prove an estimate which is necessary for the proof of Theorem \ref{MainThm} given in Section 4 below. The reader who who wishes to skip to the proof of Theorem \ref{MainThm} only needs to be familiar with the statement of Proposition \ref{cancLem} and the list of inequalities immediately following its proof. 
		
		Fix an $m$-linear CZ operator $T$ with kernel $K$. We have to suitably decompose the kernel to verify the initial assumptions of Theorem \ref{abstractThm}. Recall that we have the basic size estimate $$|K(x_1,..., x_m, x_{m+1})| \leq\frac{C_K}{(\underset{\substack{i \neq j}}{\max} \ |x_i - x_j|)^{md}},$$  
		
		\noindent where $1 \leq i,j \leq m+1$. This estimate motivates the following decomposition. Let $\vec{x} = (x_1,...,x_m, x_{m+1})$ with $x_{i} \in \R^{d}$, and define $$M_{\vec{x}} = (\underset{\substack{i \neq j}}{\max} \ |x_i - x_j|)^{md}$$ and $$M_{ij} = \{\vec{x} \in \R^{(m+1)d}: M_{\vec{x}} = |x_i - x_j| \}.$$ Note that there are $m(m-1)/2$ such sets $M_{ij}$. Also define $M_{ij}^{\ast}$ be the open subset of $\R^{(m+1)d}$ where  $$|x_i - x_j| > \frac{1}{2}|x_k - x_l|$$ for all $k,l$. Let $\psi^{ij}$ be a smooth partition of unity of $\R^{(m+1)d} \backslash \{0\}$ subordinated to the open cover $\{M_{ij}^{\ast}\}$, such that $\psi_{ij} = 1$ on $M_{ij}$. Let $K^{ij} = \psi^{ij}K$, so that our operator splits as $$T(f_1,...,f_m)(x_{m+1}) = \sum_{i,j} \int_{\R^{md}} K^{ij}(x_1, ..., x_m, x_{m+1})f_1 (x_1)...f_m (x_m) d\vec{x}.$$ Now choose the integer $l$ so that $2^{l-1} < m(m-1)/2 \leq 2^{l}$ and set $$A_{s} = \{\vec{x} \in \R^{(m+1)d} : \ 2^{s-l-1} < \sum_{i= 1}^{m}\sum_{j = i+1}^{m} |x_i - x_j| \leq 2^{s-1} \}.$$ Let $\phi$ be a smooth radial function supported in $A_0$ such that $\sum_{s}\phi(2^{-s}\vec{x}) = 1$ for $\vec{x} \neq 0$, and write $\phi_{s}(\vec{x}) = \phi(2^{-s}\vec{x}).$ Finally, let $$K_{s}^{ij}(\vec{x}) = \psi^{ij}(\vec{x})\phi_{s}(\vec{x})K(\vec{x}).$$ Notice that if $\vec{x} \in M_{ij} \cap A_{s},$ then $|x_i - x_j| \sim 2^{s}$. This leads to the decomposition \begin{align*} K^{ij}(x_1,...,x_m, x_{m+1}) &= \sum_{s} K_{s}^{ij}(x_1,...,x_m, x_{m+1}) \\ &= \sum_{s}K^{ij}(x_1,...,x_{m}, x_{m+1})\phi_{s}(\vec{x}).\end{align*}  Observe that if $|x_i- x_j| \sim 2^{s}$ in the region $M_{ij}^{\ast}$ then necessarily $|x_k - x_l| \leq 2^{s}$ for all other pairs, since $|x_k - x_l|  \leq 2|x_i - x_j|$ in $M_{ij}^{\ast}$. Given $1 \leq p < \infty$ we can easily prove the single-scale size estimate 
		
		\begin{equation}\label{LpEst} \left(\int_{\R^{md}}|K^{ij}_{s}(\vec{x})|^{p} \right)^{1/p}\lesssim 2^{-\frac{mds}{p'}}\end{equation} 
		
		\noindent for each $i,j$. This bound holds uniformly in the free variable. One also has $$\|K^{ij}_{s}\|_{L^{\infty}(\R^{md})} < 2^{-mds},$$ with a similar interpretation for the free variable. Therefore $[K]_{p} < \infty$ in this setting, for $1 \leq p \leq \infty$. We are also free to assume that $T$ satisfies the uniform truncation bound $C_T(r_1,...,r_m, \alpha) < \infty$ for some collection of exponents with $\frac{1}{r_1} + ... + \frac{1}{r_m} = \frac{1}{\alpha}$ (see \cite{MFA}, \cite{GrTo}).
		
		Assume for the rest of the section that $\mathcal{P}$ is a fixed stopping collection of cubes with top $Q$. We will work towards a proof of the estimates \eqref{assumptionL} needed to apply Theorem \ref{abstractThm} in this context. We begin by proving a single-scale estimate in a slightly more abstract setting, which will be useful for the rest of the paper.  
		
		\subsection{Single-Scale Estimates} Suppose $b = \sum_{L \in \mathcal{P}} b_L$ with $b_L$ supported in $L$. Also pick functions $g_1,...,g_{m+1}$ supported in $3Q$. The following lemma applies to all kernels $K = \sum_{s \in \Z}K_s$ satisfying the \textbf{(S)} properties from section 2. 
		
		\begin{lemma}\label{gbSS1} Suppose $b,g_1, g_2,...,g_{m+1}$ are defined as above and $1 < \beta \leq \infty.$ Let $T$ be an $m$-linear operator with kernel $K$, such that $K$ satisfies the \textbf{(S)} properties with $[K]_{\beta'} < \infty$. For a fixed $l \geq 1$ we have \begin{align*}\int_{\R^{(m+1)d}}\sum_{s}&|K_{s}|\cdot |b_{s-l}| (x_1)  |g_2(x_2)|... |g_{m+1}(x_{m+1})| \ d\vec{x} \\ &\lesssim [K]_{\beta'} |Q| \|b\|_{\mathcal{X}_{1}}\|g_2\|_{\mathcal{Y}_{\beta}}...\|g_{m+1}\|_{\mathcal{Y}_{\beta}}.\end{align*} Symmetric estimates also hold for the other tuples ($g_1, b, g_3,...,g_{m+1}$), ..., $(g_1,...,g_m, b).$ \end{lemma}

		\begin{proof} From \eqref{ssRep1}, we see that it is enough to show that if $L \in \mathcal{P}$ and $s = s_L + l$, then \begin{align*}\int_{\R^{(m+1)d}} &|K_{s}| |b_L(x_1)| |g_{2}(x_2)|... |g_{m+1}(x_{m+1})| d\vec{x} \\ &\lesssim [K]_{\beta'} |L| \|b\|_{\mathcal{X}_{1}} \|g_2\|_{\mathcal{Y}_{\beta}}...\|g_{m+1}\|_{\mathcal{Y}_{\beta}}. \end{align*} We can then sum over $L$ and use disjointness to complete the proof. Begin by fixing $x_1 \in L$, with the goal of proving an $L^{\infty}_{x_1}$ bound for $$\int_{\R^{md}}|K_{s}(x_1,..., x_{m+1})| |g_{2}(x_2)|... |g_{m+1}(x_{m+1})| dx_2...dx_{m+1}.$$ We change variables and set $z_i = x_{i} - x_{1}$ for $i = 2,...,m+1$.  Notice that the support of $K_{s}$ implies $|z_i| \lesssim 2^{s}$ for each $i$. Now for all such $x_1$,
			
			\begin{align*}&\int_{\R^{md}} |K_{s}(x_1, x_2,...,x_{m+1})||g_{2}(x_2)|... |g_{m+1}(x_{m+1})| dz_2...dz_{m+1}. \\  &\leq \left(\int_{\R^{md}}|K_{s}(x_1, z_2 + x_1,...,z_{m+1} + x_1 )|^{\beta'} dz_2...dz_{m+1} \right)^{1/\beta'} \\ & \ \ \ \ \ \ \ \ \ \ \ \  \times \prod_{i=2}^{m+1}\left( \int_{B(x_1, 2^{s + 10})}|g_{i}|^{\beta} dz_i \right)^{1/\beta}\\ &\leq 2^{-\frac{mds}{\beta}}[K]_{\beta'} \cdot \prod_{i=2}^{m+1}\left( \int_{B(x_1, 2^{s + 10})}|g_{i}|^{\beta} dz_i \right)^{1/\beta} \\ &\lesssim [K]_{\beta'} \prod_{i=2}^{m+1} \inf_{\hat{L}}M_{\beta}(g_{i}) \cdot \\ & \lesssim [K]_{\beta'} \|g_2\|_{\mathcal{Y}_{\beta}}...\|g_{m+1}\|_{\mathcal{Y}_{\beta}}. \end{align*} Note that $s > s_{L}$ and $x_1 \in L$, so $\hat{L}$ is contained in $B(x_1, 2^{s + 10})$. This justifies adding the infinums in the fourth line. As a consequence, we get \begin{align*}\int_{\R^{(m+1)d}} |K_{s}| &|b_L(x_1)| |g_{2}(x_2)|... |g_{m+1}(x_{m+1})| d\vec{x} \\ &\lesssim [K]_{\beta'} |L| \|b\|_{\mathcal{X}_{1}} \|g_2\|_{\mathcal{Y}_{\beta}}...\|g_{m+1}\|_{\mathcal{Y}_{\beta}},\end{align*} since $\|b_{L}\|_{L^{1}} \lesssim |L|\|b\|_{\mathcal{X}_{1}}$. This completes the proof of the main estimate. The relevant estimates for the tuples $(g_1, b, g_3, ..., g_{m+1})$, etc., can be proved in just the same way using the representations from Proposition \ref{adjointDecomp}. Notice that the error term $\phi$ from this proposition has an acceptable contribution.\end{proof}
		
		\noindent If $\Lambda$ is the $(m+1)$-linear form associated to an $m$-linear CZ operator $T$, then from \eqref{ssRep1} we see that \begin{equation} \Lambda_{\mathcal{P}}(b,g_{2},..., g_{m+1}) = \int_{\R^{(m+1)d}}\sum_{ij}\sum_{s}\sum_{l \geq 1}K_{s}^{ij}b_{s - l} (x_1)g_2 (x_2)... g_{m+1}(x_{m+1}) d\vec{x}.\end{equation} The number of pairs $i,j$ is bounded by a constant depending only on $m$, so it is clear that the result of Lemma \ref{gbSS1} applies to this operator as well (with a possibly different constant).

		\subsection{Cancellation Estimates} Let $T$ be an $m$-linear Calder\'{o}n-Zygmund operator as defined above, such that \begin{equation}\label{czSmooth}|K(x_1,x_2,..., x_{m+1}) - K(x',x_2,..., x_{m+1})| \leq \frac{A|x_1 - x'|^{\epsilon}}{ (\underset{i \neq j}{\max} |x_i - x_j| )^{2d + \epsilon} }\end{equation} for some $0 < \epsilon \leq 1$ whenever $|x_1 - x'| \leq \frac{1}{m+1}(|x_2 - x_1| + ... + |x_{m+1} - x_1|),$ and suppose that symmetric estimates hold with respect to the other variables $x_i$ (with the same constant). If $C_K$ is the constant from the basic Calder\'{o}n-Zygmund size estimate, we also assume that $C_K \leq A$. We then call \eqref{czSmooth} the `\textit{$\epsilon$-smoothness}' property of $T$, and $A$ the `smoothness' constant. We will use \eqref{czSmooth} to sum over $l$ in the context of Lemma \ref{gbSS1}, and as a result prove the required estimates \eqref{assumptionL} for the sparse bound. It is possible that this smoothness condition may be relaxed, but it is all we need below for the proof of Theorem \ref{MainThm}.  
		
		Suppose $b = \sum_{L\in \mathcal{P}} b_L$ and $g_1,...,g_{m+1}$ are as in the last section, with the additional property that $\int b_L = 0$ for each $L$. We will use the cancellation of $b$ and \eqref{czSmooth} to estimate $$\Lambda_{\mathcal{P}}(b, g_2,..., g_{m+1}), ..., \Lambda_{\mathcal{P}}(g_1,...,g_m, b).$$ Before proceeding, we need to prove a technical lemma related to the truncation. 
		
		\begin{lemma}\label{truncLem1} Suppose $|x_1 - x'| \leq 2^{s-l}$ for some integer $l \geq 1$, and let $$\psi_{s}^{ij}(x_1,...,x_{m+1}) = \psi^{ij}(x_1,...,x_{m+1})\phi_{s}(x_1,...,x_{m+1}).$$ Then \begin{align}\nonumber &\big| \sum_{i,j}K(x_1,x_2,...,x_{m+1})\psi_{s}^{ij}(x_1,x_2,...,x_{m+1}) \\ \nonumber & \ \ \ \ \ \ \ \  - K(x',x_2,...,x_{m+1})\psi_{s}^{ij}(x',x_2,...,x_{m+1}) \big| \\  \label{truncLem1.1} &\lesssim \sum_{i,j}\textbf{1}_{M_{ij}^{\ast}}(\vec{x})\phi_{s}(x_1,...,x_{m+1})\left|K(x_1,...,x_{m+1}) - K(x',...,x_{m+1}) \right| \\ \nonumber & \ \ \ \ \ \ \ \  + 2^{-l}\|\grad \phi_{0} \|_{\infty}|K(x',...,x_{m+1})|. \end{align} Symmetric estimates also hold for the other variables under suitable assumptions. \end{lemma}
		
		\begin{proof} We will prove the lemma for the pair $x_1, x'.$ The proof for the other variables is identical. For notational simplicity, below we will write $K = K(x_1,...,x_{m+1})$ and $\widetilde{K} = K(x',x_2,...,x_{m+1})$. After adding and subtracting $\widetilde{K}\cdot\psi^{ij}_{s}$ to each term of the sum, we can estimate \eqref{truncLem1.1} by $$\sum_{ij} |K - \widetilde{K}| \psi^{ij}_{s}(\vec{x}) + \left|\sum_{ij} (\psi_{s}^{ij}(x',x_2,...x_{m+1}) - \psi_{s}^{ij}(x_1,...,x_{m+1})) \right|\cdot |\widetilde{K}|.$$ So we have to show that $$ \left|\sum_{ij}(\psi_{s}^{ij}(x',x_2,...x_{m+1}) - \psi_{s}^{ij}(x_1,...,x_{m+1})) \right|  \lesssim 2^{-l}\|\phi_{0} '\|_{\infty}$$ when $|x_1 - x'| \leq 2^{s-l}.$ Note that under this hypothesis \begin{align*}|\phi_{s}(x_1, x_2,...,x_{m+1}) - \phi_{s}(x', x_{2},..., x_{m+1})| &\lesssim \|\grad \phi_{0} \|_{\infty} 2^{-s}|x' - x_1| \\ &\lesssim 2^{-l}\|\grad \phi_{0} \|_{\infty},\end{align*} by the scaling and smoothness of $\phi_{s}$. Summing over $i,j$ and applying this estimate proves the lemma. \end{proof}

		\noindent We can now prove the main result of the section. 
		
		\begin{prop}\label{cancLem} Let $K$ be an $m$-linear Calder\'{o}n-Zygmund kernel such that the $\epsilon$-smoothness condition \eqref{czSmooth} holds with constant $A_{\epsilon}$. Fix $1 < p \leq \infty$. Then \begin{align*}\big|\sum_{ij}\sum_{l \geq 1}\int_{\R^{(m+1)d}}&\sum_{s} K^{ij}_{s}\cdot b_{s - l} (x_1)g_2(x_2)... g_{m+1}(x_{m+1}) d\vec{x} \big| \\ &\lesssim_{\epsilon} A_{\epsilon}|Q| \|b\|_{\dot{\mathcal{X}_{1}}}\|g_2\|_{\mathcal{Y}_{p}}...\|g_{m+1}\|_{\mathcal{Y}_{p}}.\end{align*} Analogous estimates also hold for $\Lambda(g_1, b, g_3,...,g_{m+1}), ..., \Lambda(g_1,...,g_m, b).$ \end{prop} 
		
		\begin{proof} By $\eqref{ssRep1}$, it suffices to prove 
			
			\begin{align}\label{cancellation1} \sum_{l \geq c_{m}} \big| \int_{\R^{(m+1)d}}&\sum_{i,j}K^{ij}_{s_{L} + l} b_{L}(x_1)g_2(x_2)...g_{m+1}(x_{m_1})d\vec{x} \big| \\ \nonumber &\lesssim_{\epsilon} A_{\epsilon}|L| \|b\|_{\dot{\mathcal{X}_{1}}}\|g_2\|_{\mathcal{Y}_{p}}...\|g_{m+1}\|_{\mathcal{Y}_{p}},\end{align} where $c_{m}$ is some constant depending on the multilinearity parameter $m$, to be determined below. We can then use disjointness of the collection to sum over $L$, and finitely many applications of Lemma \ref{gbSS1} to handle the cases where $l < c_{m}$. Let $x'  = c_L$, the center of $L$. Note that if $s = s_L + l$ then $|x_1 - x'| \leq 2^{s-l}$ since $x_1 \in L$. We use the mean-zero condition on $b_L$ to replace $\sum_{i,j}K^{ij}_{s_{L} + l}(x_1,x_2,..., x_{m+1})$ by \begin{equation}\label{cancLem1.1} \sum_{i,j}\left(K^{ij}_{s_{L} + l}(x_1,x_2,..., x_{m+1}) - K^{ij}_{s_{L} + l}(x',x_2,..., x_{m+1})\right),\end{equation} and then estimate \eqref{cancLem1.1} using Lemma \ref{truncLem1}.  To complete the proof, we can argue as in the proof of Lemma \ref{gbSS1}, using the cancellation estimate \eqref{czSmooth} in place of the basic size estimate for $|K^{ij}_{s}|$ (with $s = s_L + l$). 
			
			We now sketch the rest of the proof. Set $F = K - \widetilde{K}$ and observe that the first term from the estimate in \eqref{truncLem1.1} contributes $$\sum_{ij} \int_{\R^{(m+1)d}} \textbf{1}_{M_{ij}^{\ast}}(\vec{x})\phi_{s}(\vec{x}) |F(x_1,...,x_{m+1}) b_L(x_1) g_2(x_2)...g_{m+1}(x_{m+1})| d\vec{x}.$$ We now fix $x_1 \in L$, with the goal of proving an $L^{\infty}_{x_1}$ bound for $$\int_{\text{supp}K^{ij}_{s} }|F(x_1,..., x_{m+1})| |g_{2}(x_2)|... |g_{m+1}(x_{m+1})| dx_2...dx_{m+1}.$$ Once again make the change of variables $z_i = x_{i} - x_{1}$ for $i = 2,...,m+1$, so in particular $|z_i| \lesssim 2^{s}$. Then we proceed as in Lemma \ref{gbSS1}, with $F(x_1, x_1 + z_2,...,x_1 + z_2 )$ in place of the $K_{s}$ term. The only change is in the estimate involving the kernel term. Since we are working at scale $\sim s$, in any region $M_{ij}^{\ast}$ there is some uniform constant $\alpha_{m}$ depending only on $m$ such that $|x_i - x_j| > 2^{s - \alpha_m}.$ Therefore we must have either $|x_1 - x_i| \geq 2^{s-\alpha_m -1 }$ or $|x_1 - x_j| \geq 2^{s- \alpha_m - 1}.$ It follows that we can choose some uniform $c_{m}$ such that if $l \geq c_{m}$ and $|x_1 - x'| \leq 2^{s-l}$, then $$|x_1 - x'| \leq \frac{1}{m + 1}(|z_2| + ... + |z_{m+1}|).$$ Hence we can apply the cancellation estimate \eqref{czSmooth} for $|F(x_1, x_1 + z_2,...,x_1 + z_{m+1})|.$ A straightforward calculation shows that for this choice of $s$ and any $i,j$ and $x_1\in L,$ we have \begin{equation}\label{cancLem1.2} \left(\int_{ \text{supp $K^{ij}_{s}(x_1, \cdot)$}}\frac{|x_1 - x'|^{\epsilon p'}}{ (|x_i - x_j|)^{2dp' + \epsilon p'} } \ dx_2...dx_{m+1} \right)^{1/p'} \ \lesssim \ 2^{-\epsilon l}2^{-\frac{mds}{p}}.\end{equation} This allows us to proceed as in the proof of Lemma \ref{gbSS1}, with an addition of the $2^{-\epsilon l}$ term. The second term from Lemma \ref{truncLem1} can be handled using the same methods from Lemma \ref{gbSS1}. Since the number of pairs $i,j$ is bounded by some $c_{m}$, we ultimately see that the left side of \eqref{cancellation1} is bounded by \begin{align*} c_{m}\sum_{l \geq 1}&(2^{-\epsilon l} + 2^{-l}\|\grad \phi_{0}\|_{\infty})A_{\epsilon}|L|  \|b\|_{\dot{\mathcal{X}_{1}}}\|g_2\|_{\mathcal{Y}_{p}}...\|g_{m+1}\|_{\mathcal{Y}_{p}} \\ &\lesssim_{\epsilon}A_{\epsilon} |L| \|b\|_{\dot{\mathcal{X}_{1}}}\|g_2\|_{\mathcal{Y}_{p}}...\|g_{m+1}\|_{\mathcal{Y}_{p}},\end{align*} and summing over $L$ proves the desired estimate. To prove the remaining part of the proposition, use the representations from Proposition \ref{adjointDecomp} and repeat the argument just given with respect to the new variables. \end{proof}
		
		\noindent As a consequence of the preceding proposition and \eqref{ssRep1}, we see that\begin{align*} |\Lambda_{\mathcal{P}}(b, g_2, g_3,..., g_{m+1})| &\leq A_{\epsilon}|Q|\|b\|_{\dot{\mathcal{X}}_{1}}\|g_2\|_{\mathcal{Y}_{p}}\|g_3\|_{\mathcal{Y}_{p}}...\|g_{m+1}\|_{\mathcal{Y}_{p}} \\  |\Lambda_{\mathcal{P}}(g_1, b, g_3, ..., g_{m+1})| &\leq A_{\epsilon} |Q|\|g_1\|_{\mathcal{Y}_{p}}\|b\|_{\dot{\mathcal{X}}_{1}}\|g_3\|_{\mathcal{Y}_p} ... \|g_{m+1}\|_{\mathcal{Y}_{p}} \\  |\Lambda_{\mathcal{P}}(g_1, g_2, b, g_4, ..., g_{m+1})| &\leq A_{\epsilon} |Q|\|g_1\|_{\mathcal{Y}_{p}}\|g_2\|_{\mathcal{Y}_{p}}\|b\|_{\dot{\mathcal{X}}_{1}}\|g_4\|_{\mathcal{Y}_{p}}...\|g_{m+1}\|_{\mathcal{Y}_{p}} \\ &\vdots \\  |\Lambda_{\mathcal{P}}(g_1, g_2, ..., g_{m}, b)| &\leq A_{\epsilon} |Q|\|g_1\|_{\mathcal{Y}_{p}}\|g_2\|_{\mathcal{Y}_{p}}...\|g_{m}\|_{\mathcal{Y}_{p}}\|b\|_{\dot{\mathcal{X}}_1}, \end{align*}  where $A_{\epsilon}$ is the smoothness constant from \eqref{czSmooth}. These estimates hold uniformly over finite truncations and stopping collections. We can therefore apply Theorem \ref{abstractThm} to prove the following. 
		
		\begin{thm}\label{CZsparse} Let $T$ be an $m$-linear Calder\'{o}n-Zygmund operator satisfying the $\epsilon$-smoothness condition \eqref{czSmooth} with constant $A_{\epsilon}$. Suppose $T$ is bounded from $L^{r_1}(\R^{d})\times...\times L^{r_m}(\R^{d}) \rightarrow L^{\alpha}(\R^{d}),$ such that $C_T(r_1,...,r_m, \alpha) < \infty$. Also fix $1 < p \leq \infty$. If $\Lambda$ is the $(m+1)$-linear form associated to $T$, then for any $f_1,..., f_{m+1}$ with $f_1 \in L^{1}(\R^{d})$ and $f_i \in L^{p}(\R^{d})$ for $i =2,...,m+1,$ we have   $$|\Lambda(f_1,f_2,...,f_{m+1})| \leq c_{d}\left[C_T + A_{\epsilon}\right] \sup_{\mathcal{S}} \text{PSF}_{\mathcal{S}}^{(1,p,...,p)}(f_1,f_2,...,f_{m+1}).$$  
		\end{thm}
		
		\noindent As we mentioned in the introduction, sparse domination results for multilinear Calder\'{o}n-Zygmund operators have been known for a few years. The main novelty here is that we do not appeal to local mean oscillation \cite{DyCalc} or maximal truncation estimates \cite{KLi}. We obtain a subset of the known weighted estimates for multilinear Calder\'{o}n-Zygmund operators as an easy corollary of Theorem \ref{CZsparse}. See \cite{MultiWeightNon} and \cite{DyCalc} for some examples.  
		
		\vspace{5mm}
		
		\noindent\textbf{Remark.} The same sparse domination result will hold if we assume an abstract Dini condition for the kernel, as in Lemma 3.2 in \cite{roughSparse}. Given $1 < p \leq \infty$, let $$\Gamma_{p}(h) = \left(\|K_s(x, x + \cdot,..., x+\cdot) - K_s(x+h, x + \cdot,..., x+\cdot)\|_{p} +...  \right).$$ The last `...' before the end of the parenthesis indicates the other possible symmetric terms. Now define $$\omega_{j,p}(K):= \sup_{s \in \Z}  \  2^{\frac{sdm}{p'}} \sup_{x \in \R^{d}}\sup_{\substack{ h \in \R^{d} \\ \|h\|_{\infty} < 2^{s-j-c_m }}} \Gamma_{p}(h).$$ Then if $[K]_{p} < \infty$ and $$[K]_{1,p} := \sum_{j=1}^{\infty}\omega_{j,p}(K) < \infty,$$ we can prove an analogue of Proposition \ref{cancLem} and show that the assumptions of Theorem \ref{abstractThm} are satisfied. The argument is similar to the proof of Lemma 3.2 in \cite{roughSparse}. However, note that we would have to do essentially the same amount of work as above to check that this condition is satisfied by the $\epsilon$-smooth Calder\'{o}n-Zygmund operators considered in this section. 
		
		\section{Rough Bilinear Singular Integrals}
		
		\indent
		
		In this section we use our sparse domination theorem to prove Theorem \ref{MainThm}. Recall that $$T_{\Omega}(f_{1}, f_{2})(x) = \textrm{p.v. } \int_{\R^{d}}\int_{\R^{d}} f_{1}(x - y_{1})f_{2}(x - y_{2})\frac{\Omega((y_1, y_2)/|(y_1, y_2)|)}{|(y_1, y_2)|^{2d}} \ dy_1 dy_2,$$ with $\Omega \in L^{\infty}(S^{2d-1})$ and $\int_{S^{2d-1}} \Omega = 0$. We will use the results from the paper \cite{BiRough} by Grafakos, He, and Honz\'{\i}k. We quickly review their initial decomposition of the kernel $K$ of $T_{\Omega}$. Let $\{\beta_{j}\}_{j \in \Z}$ be a smooth partition of unity on $\R^{2d} \backslash \{0\}$, with $\beta_{j}$ adapted to the annulus $\{2^{j-1} < |z| < 2^{j+1}\}$. Let $\Delta_{j}$ denote the standard Littlewood-Paley operator localizing frequencies at scale $2^{j}$. Then define $K_{j}^{i} = \Delta_{j-i}\beta_{i}K$ and decompose $K = \sum_{j\in \Z}K_{j},$ with $$K_{j} = \sum_{i \in \Z}K_{j}^{i}.$$ Write $T^{j}$ for the operator associated to the kernel $K_{j}$. 
		
		\begin{prop}[Prop. 3 and 5 in \cite{BiRough}]\label{GrafThm} There exists $0 <\delta < 1$ so that $$\|T_{j}(f_1,f_2)\|_{L^{1}} \leq C2^{-\delta j}\|\Omega\|_{\infty}\|f\|_{L^2} \|g\|_{L^2}$$ when $j \geq 0,$ and $$\|T_{j}(f_1,f_2)\|_{L^{1}} \leq C2^{-(1-\delta)|j|}\|\Omega\|_{\infty}\|f\|_{L^2} \|g\|_{L^2}$$ when $j < 0$. 
		\end{prop}
		
		\begin{prop} [Lem. 10 in \cite{BiRough}] \label{GrafCZ} Fix any $j \in \Z$ and $0< \eta < 1$. The kernel of $T_{j}$ is a bilinear Calder\'{o}n-Zygmund kernel that satisfies the $\eta$-smoothness condition \eqref{czSmooth} with constant $A_{j,\eta} \leq C_{d,\eta} \|\Omega \|_{\infty} 2^{|j|\eta}.$
		\end{prop}
		
		\noindent By Proposition \ref{GrafCZ} we can apply Proposition \ref{cancLem} at each scale $j$, with resulting constant $A_{j,\eta} \leq C_{d,\eta} \|\Omega \|_{\infty} 2^{|j|\eta}.$ The $\eta$ can be arbitrarily small, so we will be able to interpolate with the bound from Proposition \ref{GrafThm} and then sum over $j$ to get the required estimates for the sparse-form bound of $T_{\Omega}$. 
		
		\subsection{Interpolation Lemmas}
		
		Assume below that a stopping collection $\mathcal{P}$ with top $Q$ has been fixed. The following lemmas allow us to interpolate bounds between various $\mathcal{X}_{\cdot}(\mathcal{P})$ and $\mathcal{Y}_{\cdot}(\mathcal{P})$ spaces in the trilinear setting. A more general interpolation theorem for these spaces should be available, but we only prove two particular results needed for the proof of Theorem \ref{MainThm}.  
		
		\begin{lemma}\label{interpolation} Fix any $0 < \epsilon < \frac{1}{8}$ and suppose $p = 1 + 4\epsilon$. Let $\Lambda$ be a (sub)-trilinear form such that $$|\Lambda(b,g,h)| \leq A_1 \|b\|_{\dot{\mathcal{X}_{1}}} \|g\|_{\mathcal{Y}_p} \|h\|_{\mathcal{Y}_p}$$ and $$ |\Lambda(b,g,h)| \leq A_2 \|b\|_{\dot{\mathcal{X}_{2}}} \|g\|_{\mathcal{Y}_2} \|h\|_{\mathcal{Y}_{\infty}}.$$  Then if $q = p + 4\epsilon$ and $q < r < \infty$, we have $$|\Lambda(f_1,f_2,f_3)| \leq (A_1)^{1-\epsilon}(A_2)^{\epsilon}\|f_1\|_{\dot{\mathcal{X}_{q}}} \|f_2\|_{\mathcal{Y}_q} \|f_3\|_{\mathcal{Y}_{r}}.$$
		\end{lemma}  
		
		\begin{proof} We can assume $A_2 < A_1.$ We also make the normalizations $A_1 = 1$ and $$\|f_1\|_{\dot{ \mathcal{X}_q }} = \|f_2\|_{\mathcal{Y}_q } = \|f_3\|_{ \mathcal{Y}_r} = 1.$$ It will now be enough to prove that $$|\Lambda(f_1,f_2,f_3)| \lesssim A_{2}^{\epsilon}.$$ Pick $\lambda \geq 1$ and define $f_{> \lambda} = f\textbf{1}_{|f| > \lambda}$. We decompose $f_1 = b_1 + g_1,$ where $$b_1 := \sum_{R \in \mathcal{P}} \left( (f_1)_{> \lambda} - \frac{1}{|R|}\int_{R}(f_1)_{>\lambda} \right)\textbf{1}_{R}.$$ For $i = 2,3$ we also decompose $f_i = b_i + g_i$ with $b_i := (f_i)_{> \lambda}.$  Using the normalization assumption and the definition of the $\mathcal{Y}$ spaces, it is straightforward to verify the following estimates: 
			
			\begin{align}\label{interpolationEst} &\|g_1\|_{\dot{\mathcal{X}_{p}}} \lesssim 1, \  \|g_1\|_{\dot{\mathcal{X}_{2}}} \lesssim \lambda^{1 - \frac{q}{2}}, \ \ \ \|b_1\|_{\dot{\mathcal{X}_1}} \lesssim \lambda^{1-q} \\ \nonumber &\|g_2\|_{\mathcal{Y}_{2}} \lesssim \lambda^{1 - \frac{q}{2}}, \ \ \ \|b_2\|_{\mathcal{Y}_p} \lesssim \lambda^{1-\frac{q}{p}} \\ \nonumber &\|g_3\|_{\mathcal{Y}_{\infty}} \lesssim \lambda, \ \ \ \|b_3\|_{\mathcal{Y}_p} \lesssim \lambda^{1 - \frac{r}{p}}.\end{align} 
			
			\noindent Note that $g_1 \in \dot{\mathcal{X}_{2}}$ since $f_1$ and $b_1$ are supported on cubes in $\mathcal{P}$, with mean zero on each cube. We now estimate $|\Lambda(f_1, f_2, f_3)|$ by the sum of eight terms \begin{align*} |\Lambda(g_1, g_2, g_3)| + |\Lambda(b_1, g_2, g_3)| + |\Lambda(g_1, b_2, g_3)| + |\Lambda(g_1, g_2, b_3)| \\ + \ |\Lambda(b_1, b_2, g_3)| + |\Lambda(b_1, g_2, b_3)| + |\Lambda(g_1, b_2, b_3)| + |\Lambda(b_1, b_2, b_3)|.
			\end{align*} Applying the estimates in \eqref{interpolationEst}, we get $$|\Lambda(g_1,g_2, g_3)| \lesssim A_2 \|g_1\|_{\dot{\mathcal{X}_{2}}} \|g_2\|_{\mathcal{Y}_2} \|g_3\|_{\mathcal{Y}_{\infty}} \lesssim A_2 \lambda^{3 - q}.$$ For the remaining seven terms, we use the first assumption $|\Lambda(b,g,h)| \leq \|b\|_{\dot{\mathcal{X}_{1}}} \|g\|_{\mathcal{Y}_p} \|h\|_{\mathcal{Y}_p},$ along with suitable estimates from \eqref{interpolationEst}. Here it is crucial that $g_1, b_1 \in \dot{\mathcal{X}_{2}}.$ Ultimately we conclude 
			
			\begin{align*}|\Lambda(f_1, f_2, f_3)| \lesssim \lambda^{1 - q} +& \lambda^{1 - \frac{q}{p}} + \lambda^{1 - \frac{r}{p}} + \lambda^{1- q}\lambda^{1 - \frac{q}{p}} + \lambda^{1- q}\lambda^{1 - \frac{r}{p}} + \lambda^{1 - \frac{q}{p}}\lambda^{1 - \frac{r}{p}} \\ &+ \lambda^{1- q}\lambda^{1 - \frac{q}{p}}\lambda^{1 - \frac{r}{p}} + A_2 \lambda^{3 - q}.\end{align*} 
			
			\noindent Let $\alpha = 1 - \frac{q}{p}$ and notice that $-4\epsilon < \alpha < -\epsilon$ and $\alpha = \max(1 - q, 1 - \frac{q}{p}, 1 - \frac{r}{p}).$ We can simplify the previous estimate to get \begin{align*}|\Lambda(f_1, f_2, f_3)| &\lesssim 3\lambda^{\alpha} + 3\lambda^{2 \alpha} + \lambda^{3 \alpha} + A_{2}\lambda^{3 - q} \\ &\lesssim \lambda^{\alpha}\left(7 + A_2 \lambda^{3 - q - \alpha }  \right) \\ & \lesssim \lambda^{-\epsilon}\left(7 + A_2 \lambda^{3 - q + 4\epsilon} \right). \end{align*} We set $\lambda = A_2^{-3 + q - 4\epsilon}$. Since $A_2 < 1$ and $q < 2$ this is admissible and we conclude that $$ |\Lambda(f_1, f_2, f_3)| \lesssim A_{2}^{(-3 + q - 4\epsilon)(-\epsilon)} \lesssim A_{2}^{3\epsilon + 4\epsilon^{2} - q\epsilon} \lesssim A_{2}^{\epsilon}.$$ This completes the proof. \end{proof}
		
		\begin{lemma}\label{interpolation2} Fix any $0 < \epsilon < \frac{1}{8}$ and set $p = 1 + 4\epsilon$. Let $\Lambda$ be a (sub)-trilinear form such that $$|\Lambda(g,h,b)| \leq A_1  \|g\|_{\mathcal{Y}_p} \|h\|_{\mathcal{Y}_p}\|b\|_{\dot{\mathcal{X}_{1}}}$$ and $$ |\Lambda(g,h,b)| \leq A_2 \|g\|_{\mathcal{Y}_2} \|h\|_{\mathcal{Y}_2} \|b\|_{\mathcal{Y}_{\infty}}.$$  Then if $q = p + 4\epsilon$ and $q < r < \infty$, we have $$|\Lambda(f_1,f_2,f_3)| \leq (A_1)^{1-\epsilon}(A_2)^{\epsilon}\|f_1\|_{\mathcal{Y}_{q}} \|f_2\|_{\mathcal{Y}_q} \|f_3\|_{\dot{\mathcal{X}_{r}}}.$$ \end{lemma}
		
		\begin{proof} The argument is almost identical to the proof of Lemma \ref{interpolation}. In this case we decompose $f_3 = g_3 + b_3$ with $$b_3 = \sum_{R \in \mathcal{P}} \left((f_3)_{>\lambda} - \frac{1}{|R|}\int_{R} (f_3)_{> \lambda} \right)\textbf{1}_R.$$  For $i = 1,2$ we let $f_i = g_i + b_i$ with $g_i = (f_i)_{\leq \lambda},$ and then proceed as in Lemma \ref{interpolation}. \end{proof}
		
		\subsection{The Sparse Bound}
		
		\noindent We can now prove Theorem \ref{MainThm}. Let $T_{\Omega}$ be a rough bilinear operator defined as above, with $\Omega \in L^{\infty}(S^{2d-1})$. Such operators $T_{\Omega}$ satisfy the standard Calder\'{o}n-Zygmund size estimate, so the single-scale properties \textbf{(S)} hold. The proof that $C_{T_{\Omega}} < \infty$ for some tuple of exponents is also straightforward. In particular, inspection of the proofs of Propositions \ref{GrafThm} and \ref{GrafCZ} in \cite{BiRough} shows that the same estimates hold for $T_j$ if one replaces the kernel of $T_j$ by $$\widetilde{K}_{j} := \sum_{i = -\infty}^{\infty} a_i \cdot K_{j}^{i},$$ with $a_i \in \{0, 1\}$ but otherwise arbitrary. Moreover, these estimates are uniform over the choice of $a_i$. Any truncation of the kernel of $T$ leads to a special case of this small modification, and therefore the $L^{2}\times L^{2} \rightarrow L^{1}$ estimate of Proposition \ref{GrafThm} still holds. Likewise, the Calder\'{o}n-Zygmund smoothness estimate from Proposition \ref{GrafCZ} holds uniformly over the choice of $a_i$. This is once again obvious from the proof in \cite{BiRough}. It follows that $C_{T_{\Omega}}(2,2,1) < \infty$. We must now verify the estimates \eqref{assumptionL} for $(q_1,q_2,q_3) = (r,r,r)$, where $r > 1.$

		Fix $0 < \eta < 1$. Since $T^{j}$ is a bilinear Calder\'{o}n-Zygmund kernel satisfying the $\eta$-smoothness condition \eqref{czSmooth} with constant $A \leq C_{d,\eta} \|\Omega \|_{\infty} 2^{|j|\eta},$ we see from Proposition \ref{cancLem} that $$|\Lambda^{j}(b, g, h)| \leq C_{d,\eta,p} \|\Omega \|_{\infty} 2^{|j|\eta} |Q| \|b\|_{\dot{\mathcal{X}}_{1}} \|g\|_{\mathcal{Y}_p}\|h\|_{\mathcal{Y}_p}$$ for any $1 < p \leq \infty$. From Proposition \ref{GrafThm} we also have $$|\Lambda^{j}(b, g,h)| \lesssim 2^{-c|j|}\|\Omega\|_{\infty}|Q|\|b\|_{\dot{\mathcal{X}}_{2}} \|g\|_{\mathcal{Y}_2}\|h\|_{\mathcal{Y}_{\infty}},$$ where $c$ is some positive constant independent of $j$. Interpolating via Lemma \ref{interpolation}, we find that for any $0 < \epsilon < \frac{1}{8}$ there are $1 < q < r$ so that \begin{align*} |\Lambda^{j}(b,g,h)| &\lesssim_{\epsilon,\eta} (2^{\eta|j|})^{1-\epsilon}(2^{-c|j|})^{\epsilon}|Q|\|\Omega\|_{\infty}\|b\|_{\dot{\mathcal{X}_{q}}} \|g\|_{\mathcal{Y}_q} \|h\|_{\mathcal{Y}_{r}} \\ &\lesssim_{\epsilon, \eta} 2^{\eta |j|} 2^{-\eta \epsilon |j|}2^{-c\epsilon |j|} |Q|\|\Omega\|_{\infty}\|b\|_{\dot{\mathcal{X}_{q}}} \|g\|_{\mathcal{Y}_q} \|h\|_{\mathcal{Y}_{r}}.\end{align*} Suppose we have chosen $\eta$ so that $\eta < \frac{c}{8}$. If we let $\epsilon = \frac{\eta}{c}$ then $0 < \epsilon < \frac{1}{8}$, and therefore $$ |\Lambda^{j}(b,g,h)| \lesssim_{\epsilon, \eta} 2^{-\eta \epsilon |j|} |Q|\|\Omega\|_{\infty}\|b\|_{\dot{\mathcal{X}_{q}}} \|g\|_{\mathcal{Y}_q} \|h\|_{\mathcal{Y}_{r}}.$$ This is summable over $j \in \Z$. Hence if $\Lambda_{\mathcal{P}}$ is the form associated to $T_{\Omega}$, truncated to some finite number of scales, we can conclude that \begin{equation}\label{RoughEst1}|\Lambda_{\mathcal{P}}(b,g,h)| \lesssim |Q|\|\Omega\|_{\infty}\|b\|_{\dot{\mathcal{X}_{q}}} \|g\|_{\mathcal{Y}_q} \|h\|_{\mathcal{Y}_{r}}. \end{equation}By symmetry, the same argument also yields \begin{equation}\label{RoughEst2}|\Lambda_{\mathcal{P}}(g,b,h)| \lesssim |Q|\|\Omega\|_{\infty}\|g\|_{\mathcal{Y}_q} \|b\|_{\dot{\mathcal{X}_{q}}}  \|h\|_{\mathcal{Y}_{r}}.\end{equation} These estimates are uniform over all finite truncations and stopping collections.  Finally, we argue as above and interpolate using Lemma \ref{interpolation2} to prove \begin{equation}\label{RoughEst3}|\Lambda_{\mathcal{P}}(g,h,b)| \lesssim |Q|\|\Omega\|_{\infty} \|g\|_{\mathcal{Y}_q} \|h\|_{\mathcal{Y}_{q}}\|b\|_{\dot{\mathcal{X}_{r}}},\end{equation} which is once again uniform over truncations and stopping collections.  
		
		The estimates \eqref{RoughEst1}, \eqref{RoughEst2}, \eqref{RoughEst3} show that the form associated to a rough bilinear operator $T_{\Omega}$ satisfies assumption \eqref{assumptionL} of Theorem \ref{abstractThm} with tuple $(r,r,r)$, since the $\mathcal{Y}_{q}$ norm is increasing in $q$. It is clear that we can take any $r > 1$, so this completes the proof of Theorem \ref{MainThm}. 
		
		\section{Weighted Estimates}
		
		We now prove the weighted estimates claimed in Corollary \ref{WeightedEst} and Corollary \ref{MultiWeight}. We assume that the reader is familiar with basic results from the theory of $A_p$ weights, for example the openness property \begin{equation}\label{open}[w]_{A_{p - \eta}} \lesssim [w]_{A_p} \ \text{  when } \eta = c_{d,p}[w]_{A_p}^{1-p'}\end{equation} (see \cite{CFA} or \cite{Stein} for a proof). Below we always assume that $\Omega \in L^{\infty}(S^{2d-1})$ with mean zero. 
		
		\subsection{Single Weight - Proof of Corollary \ref{WeightedEst}}
		
		The argument is a combination of known results from the sparse domination theory. It suffices to prove the desired estimate for fixed $p > 2$. We can then apply the multilinear extrapolation theory due to Grafakos and Martell, in particular Theorem 2 in \cite{GrMart}. Fix $2 < p < \infty$ and $w \in A_p,$ with the goal of showing that $$\|T_{\Omega}(f,g)\|_{L^{p/2}(w)} \leq C_{p,\Omega} c_w \|f\|_{L^{p}(w)} \|g\|_{L^{p}(w)}.$$ Define $$\sigma = w^{-\frac{2}{p-2}}$$ and notice that $$w^{\frac{2}{p}}\sigma^{\frac{p-2}{p}} = 1.$$ Let $r = (p/2)' = \frac{p}{p-2}$, and choose  $q_i = 1+\epsilon_{i}$ for $\epsilon_{i}>0$, such that $q_i < r$ and $q_i < p$. By Theorem \ref{MainThm} and duality considerations, it will be enough to show that for any sparse collection $\mathcal{S}$, \begin{equation}\label{weightEst1}\text{PSF}_{\mathcal{S}}^{(q_1,q_2,q_3)}(f,g,h) \lesssim_{p} c_w \|f\|_{L^{p}(w)}\|g\|_{L^{p}(w)}\|h\|_{L^{r}(\sigma)}.\end{equation} We begin by quoting the estimate $$\text{PSF}_{\mathcal{S}}^{(1,1,q_3)}(g_1, g_2, g_3) \lesssim_{p} \gamma_w^{\max(\frac{p}{p-1} , \frac{p}{2} )}\|g_1\|_{L^{p}(w)}\|g_2\|_{L^{p}(w)}\|g_3\|_{L^{r}(\sigma)}$$ with $$\gamma_w = \sup_{Q}\left(\frac{1}{|Q|}\int_{Q}w^{-\frac{1}{p-1}} \right)^{2-\frac{2}{p}}\left(\frac{1}{|Q|}\int_{Q}\sigma^{\frac{q_{3}}{q_{3}-r}} \right)^{\frac{1}{q_3}-\frac{1}{r}}.$$ This can be proved using techniques from \cite{MultiSparse} or \cite{DyCalc} (see, for example, Theorem 3 in \cite{MultiSparse} and its proof). Now write the term involving $\sigma$ as $$\left(\frac{1}{|Q|}\int_{Q}w^{1+\alpha} \right)^{\frac{1}{q_3} - \frac{1}{r}}, \text{      } 1+\alpha = \frac{1- \frac{1}{r}}{ \frac{1}{q_3}  - \frac{1}{r}}.$$  We can apply the reverse H\"{o}lder inequality if we choose $\epsilon_3$ correctly (as we are free to do), and after some straightforward calculation this leads to the estimate $$\gamma_{w} \lesssim_{p}c_{w,p}[w]_{A_{p}}^{\frac{2}{p}},$$ where $c_{w,p}$ is a power of the constant appearing in the reverse H\"{o}lder inequality. Hence
		
		\begin{equation} \label{PSFInit}\text{PSF}_{\mathcal{S}}^{(1,1,q_3)}(g_1, g_2, g_3) \lesssim_{p}c_{w,p}[w]_{A_{p}}^{\frac{2}{p} \cdot \max(\frac{p}{p-1} , \frac{p}{2})}\|g_1\|_{L^{p}(w)}\|g_2\|_{L^{p}(w)}\|g_3\|_{L^{r}(\sigma)},\end{equation} 
		and $c_{w,p}$ can be computed explicitly in terms of $[w]_{A_p}$ (see, for example, Chapter 7.2 in \cite{CFA}). We can lift \eqref{PSFInit} to a bound for $\text{PSF}_{\mathcal{S}}^{q_1,q_2,q_3}(f,g,h)$ using an inequality due to Di Plinio and Lerner (\cite{DPLerner}, Proposition 4.1): \begin{equation}\label{DPLerner} (f)_{1+\epsilon, Q} \leq (f)_{1,Q} + 2^{d}\epsilon(M_{1+\epsilon}f)_{1,Q}. \end{equation} Recall also that $$\|M_{q}\|_{L^{p}(w) \rightarrow L^{p}(w) } \lesssim [w]_{A_{\frac{p}{q}}}^{\frac{q}{p-q}}$$ when $p > q$ (see \cite{Buck}). Using the openness of the $A_{t}$ classes, we see that if $\epsilon_1,\epsilon_2$ are chosen properly then in fact 
		\begin{equation} \label{maxWeight} \|M_{q_i}\|_{L^{p}(w) \rightarrow L^{p}(w) } \lesssim [w]_{A_{p}}^{\left(\frac{q_i}{p-q_i}\right)\left( \frac{1}{1-p} \right) }\end{equation} for $i = 1,2$. Now apply \eqref{DPLerner} to the $q_1$- and $q_2$-averages occurring in the form $\text{PSF}^{(q_1,q_2,q_3)}_{\mathcal{S}}(f,g,h)$ to get \begin{align*}\text{PSF}^{(q_1,q_2,q_3)}(f,g,h) \lesssim \text{PSF}&^{(1,1,q_3)}(f,g,h) + c_w'\text{PSF}^{(1,1,q_3)}(f, M_{1+\epsilon_2}g, h) \\ &+ \  c_w'\text{PSF}^{(1,1,q_3)}(M_{1+\epsilon_1}f, g, h) \\ &+ c_w'\text{PSF}^{(1,1,q_3)}(M_{1+\epsilon_1} f, M_{1+\epsilon_2}g, h), \end{align*} with $c_w'$ appearing from \eqref{DPLerner} and the openness estimate. Then \eqref{PSFInit} and \eqref{maxWeight} imply the claimed boundedness. If $\alpha_{w}$ is the constant appearing in \eqref{PSFInit}, we see as a consequence that $T_{\Omega}:L^{p}(w) \times L^{p}(w) \rightarrow L^{p/2}(w)$ with constant $$C_{w} \lesssim (\alpha_{w}+c_w')\cdot[w]_{A_{p}}^{\left(\frac{q_1}{p-q_1}\right)\left( \frac{1}{1-p} \right) }[w]_{A_{p}}^{\left(\frac{q_2}{p-q_2}\right)\left( \frac{1}{1-p} \right) }.$$ Note that $c_{w}'$ can also be computed explicitly in terms of $[w]_{A_p}$, after solving for the $\epsilon_1, \epsilon_2$ used above and using \eqref{open}.
		
		\subsection{Multiple Weights} We prove Corollaries \ref{MultiWeight} and \ref{MultiWeight2}. Suppose $v_1,v_2,v_3$ are strictly positive functions such that $$\prod_{i=1}^{3} v_{i}^{\frac{1}{q_i}} =1$$ for some $q_i \geq 1$ with $\sum_{i=1}^{3}\frac{1}{q_i} = 1$. Let $\vec{p} = (p_1,p_2,p_3)$ be any tuple of exponents with $1 < p_i < q_i$, and recall that $$[\vec{v}]_{A_{\vec{q}}^{\vec{p}}} = \sup_{Q}\prod_{i=1}^{3} \left(\frac{1}{|Q|}\int_{Q}v_{i}^{ \frac{p_i}{p_i - q_i} }\right)^{\frac{1}{p_i} - \frac{1}{q_i}},$$ with the supremum taken over cubes $Q \subset \R^{d}$. We can argue as in the proof of Lemma 6.1 in \cite{MultiSparse} to prove that for any sparse collection $\mathcal{S}$ we have $$\text{PSF}_{\mathcal{S}}^{\vec{p}}(f_1,f_2,f_3) \leq \left(c_{\vec{p}, \vec{q}} [\vec{v}]_{A_{\vec{q}}^{\vec{p}}}^{\max \left\{ \frac{q_i}{q_i-p_i} \right\}} \right) \prod_{i=1}^{3}\|f_i\|_{L^{q_i}(v_i)}$$ when $[\vec{v}]_{A^{\vec{p}}_{\vec{q}}}$ is finite. Applying Theorem \ref{MainThm} yields the following result. 
		
		\begin{lemma}\label{MultiWeightLem} Let $v_1,v_2,v_3$ and $\vec{p},\vec{q}$ be as above. Also assume $[\vec{v}]_{A_{\vec{q}}^{\vec{p}}} < \infty$. Then $$|\langle T_{\Omega}(f_1,f_2), f_3\rangle | \leq C_{\Omega, \vec{p}, \vec{q}} [\vec{v}]_{A_{\vec{q}}^{\vec{p}}}^{\max \left\{ \frac{q_i}{q_i-p_i} \right\}}  \prod_{i=1}^{3}\|f_i\|_{L^{q_i}(v_i)}$$ for $f_i \in L^{q_i}(v_i)$. \end{lemma} 
		
		\noindent The reader can consult \cite{MultiSparse} for the explicit value of the constant. Corollary \ref{MultiWeight2} immediately follows from this lemma by duality. We now show that Corollary \ref{MultiWeight} is also a consequence.  
		
		By using multilinear extrapolation techniques from \cite{Duo2}, \cite{GrMart}, it is enough to prove the claimed estimate of Corollary \ref{MultiWeight} when $q_i = 3$ for each $i=1,2,3$. Write $\vec{3} = (3,3,3)$, and suppose $\vec{p} = (1+\epsilon, 1+\epsilon, 1+\epsilon)$ for some small $\epsilon > 0$. Observe that if $t >0$ is chosen properly, then there is some absolute $C >0$ such that \begin{equation} \label{RHest} [v_i^{1+t}]_{A_3} \leq C[v_i]_{A_3}, \ \ \text{    } i=1,2.\end{equation}  This estimate can be proved using the openness property and reverse H\"{o}lder estimates as in the last section; see \cite{CMP}, Section 3.7, for a more general version of this inequality. Now Lemma \ref{MultiWeightLem} implies that $$ T_{\Omega}: L^{3}(v_1) \times L^{3}(v_2) \rightarrow L^{3/2}(v_{3}^{-1/2})$$ with operator norm bounded by $C_{\epsilon}[\vec{v}]^{3}_{A_{\vec{3}}^{\vec{p} }}.$ Hence it will be enough to prove \begin{equation}\label{RHest2}[\vec{v}]_{A_{\vec{3}}^{\vec{p} }} \leq [v_1^{1+t}]_{A_3}^{\frac{1}{3(1+t)}} \cdot [v_2^{1+t}]_{A_3}^{\frac{1}{3(1+t)}}\end{equation} with $1 +t = \frac{2(1+\epsilon)}{2 - \epsilon}$, since we can then apply \eqref{RHest} after choosing $\epsilon$ correctly (as we are free to do). One can prove \eqref{RHest2} by using H\"{o}lder's inequality; we omit the details.

		\section{Proof of Theorem \ref{abstractThm}} Here we prove our multilinear sparse domination theorem, and also resolve the technical issue related to multiplication operators discussed in Remarks 2.1 and 2.3 above. 
		
		\subsection{Construction of the Sparse Collection} Suppose we are given a form $\Lambda$ as in the statement of Theorem \ref{abstractThm} with $C_T = C_T(r_1,...,r_m, \alpha)$, and $\Lambda$ truncated to a finite (but otherwise arbitrary) number of scales. Given $f_{i} \in L^{p_{i}}(\R^{d})$ with compact support, we would like to construct a sparse collection of cubes $\mathcal{S}$ so that $$|\Lambda(f_1,...,f_{m+1})| \leq c_{d}\left[C_L + C_T\right]\sum_{R \in \mathcal{S}} |R|\prod_{i=1}^{m+1} (f_i)_{p_i, R}.$$ We will generalize the iterative argument from \cite{roughSparse}. The following lemma is crucial for the induction step.

		\begin{lemma}\label{sparseInit} Let $Q$ be a fixed dyadic cube and $\mathcal{P}$ a stopping collection with top $Q$. If the estimates \eqref{assumptionL} hold, then \begin{align*}|\Lambda^{s_{Q}}(h_1\textbf{1}_{Q}, h_2 \textbf{1}_{3Q}, ..., h_{m+1}\textbf{1}_{3Q})| \leq & \ C_{d}|Q|\|h_1\|_{\mathcal{Y}_{p_1}} \|h_2\|_{\mathcal{Y}_{p_2}} ... \|h_{m+1}\|_{\mathcal{Y}_{p_{m+1}}} \\ & \ \ \ + \sum_{\substack{L \in \mathcal{P}\\ L \subset Q }}|\Lambda^{s_{L}}(h_1\textbf{1}_{L}, h_2 \textbf{1}_{3L},..., h_{m+1}\textbf{1}_{3L})|.\end{align*} \end{lemma}
		
		\begin{proof} We can assume supp$h_1 \subset Q$ and supp$h_i \subset 3Q$ for $2 \leq i \leq m+1$. By definition of $\Lambda_{\mathcal{P}}$ we have \begin{align*}\Lambda^{s_{Q}}(h_1, h_2,..., h_{m+1}) = \Lambda_{\mathcal{P}}&(h_1, h_2,..., h_{m+1}) \\ &+ \sum_{\substack{L \in \mathcal{P}\\ L \subset Q }}\Lambda^{s_{L}}(h_1\textbf{1}_{L}, h_2 \textbf{1}_{3L},..., h_{m+1}\textbf{1}_{3L}),\end{align*} so it is enough to estimate $\Lambda_{\mathcal{P}}(h_1, h_2, ..., h_{m+1})$.  For each $j$ perform the Calder\'{o}n-Zygmund (CZ) decomposition \eqref{CZdecomp} of $h_j$ with respect to $\mathcal{P}$. Then $h_j = b_j + g_j$ with $$b_j = \sum_{L \in \mathcal{P}}b_{jL}, \ \ \ b_{jL} := \left( h_j - (h_j)_{L} \right)\textbf{1}_L$$ and $$\|g_{j}\|_{\mathcal{Y}_{\infty}} \lesssim \|h_j\|_{\mathcal{Y}_{p_j}}, \ \ \ \|b_{j}\|_{\dot{\mathcal{X}}_{p_j}} \lesssim \|h_j\|_{\mathcal{Y}_{p_j}}.$$ We now decompose $\Lambda_{\mathcal{P}}(h_1, h_2,..., h_{m+1})$ using the CZ decomposition. If we let $f_{h_{i}}(0) = g_i$ and $f_{h_{i}}(1) = b_i$, we see that the form breaks up into the $2^{m+1}$ terms \begin{equation} \label{czDecomp}\sum_{0\leq k_1, ..., k_{m+1} \leq 1} \Lambda_{\mathcal{P}}(f_{h_{1}}(k_1), f_{h_{2}}(k_2), ..., f_{h_{m+1}}(k_{m+1})).\end{equation}
			
			\noindent From the definition of $\Lambda_{\mathcal{P}}$ and the pairwise disjointness of elements of $\mathcal{P}$, we see that finiteness of $C_T$ implies \begin{align*}|\Lambda_{\mathcal{P}}(g_1, g_2,... g_{m+1})| &\leq C_{T}\|g_1\|_{L^{r_1}}...\|g_m\|_{L^{r_m}}\|g_{m+1}\|_{L^{\alpha'}} \\ & \ \ \ \ \ \ \ \ \ \ \ \  +  C_{T}\sum_{\substack{L \in \mathcal{P} \\ L \subset Q}} \|g_1 \textbf{1}_L\|_{L^{r_1}}...\|g_{m+1} \textbf{1}_{3L}\|_{L^{\alpha'}} \\ &\leq c_{d}C_{T}|Q| \|g_1\|_{\mathcal{Y}_{r_1}} \|g_2\|_{\mathcal{Y}_{r_2}}... \|g_{m+1}\|_{\mathcal{Y}_{\alpha'}} \\ &\leq c_{d}C_{T}|Q|\|h_1\|_{\mathcal{Y}_{p_1}} \|h_2\|_{\mathcal{Y}_{p_2}}... \|h_{m+1}\|_{\mathcal{Y}_{p_{m+1}}}.\end{align*} Here we've used $$\frac{1}{r_1} + ... + \frac{1}{r_{m}} + \frac{1}{\alpha'} = \frac{1}{\alpha} + \frac{1}{\alpha'} = 1,$$ the disjointness of $L \in \mathcal{P}$, and the $\|g_j\|_{\mathcal{Y_\infty}}$ bound. We can now use \eqref{assumptionL} and the CZ estimates to control the remaining terms by the desired quantity. For example, 
			
			\begin{align*} |\Lambda_{\mathcal{P}}(b_1, g_2,..., g_{m+1})| \leq C_L |Q| \|b_1\|_{\dot{\mathcal{X}}_{p_1}}\|g_2\|_{\mathcal{Y}_{p_2}}...\|g_{m+1}\|_{\mathcal{Y}_{p_{m+1}}} \\ \leq c_{d}C_L |Q| \|h_1\|_{\mathcal{Y}_{p_1}}\|h_2\|_{\mathcal{Y}_{p_2}}...\|h_{m+1}\|_{\mathcal{Y}_{p_{m+1}}}\end{align*} 
			
			\noindent by the first estimate in \eqref{assumptionL} and the CZ properties. The other estimates are similar. We use the last estimate in \eqref{assumptionL} to control the single term $\Lambda_{\mathcal{P}}(g_1, g_2,...,g_m, b_{m+1})$: \begin{align*}|\Lambda_{\mathcal{P}}(g_1, g_2,..., g_{m}, b_{m+1})| \leq C_L|Q|\|g_1\|_{\mathcal{Y}_{\infty}}\|g_2\|_{\mathcal{Y}_{\infty}}...\|g_m\|_{\mathcal{Y}_{\infty}}\|b_{m+1}\|_{\dot{\mathcal{X}}_{p_{m+1}}} \\ \leq c_d C_L|Q|\|h_1\|_{\mathcal{Y}_{p_1}}\|h_2\|_{\mathcal{Y}_{p_2}}...\|h_m\|_{\mathcal{Y}_{p_m}}\|h_{m+1}\|_{\mathcal{Y}_{p_{m+1}}}.\end{align*} To bound the remaining terms, we choose from the first $m$ inequalities in \eqref{assumptionL}. If there are more than one $b_i$ terms in the particular piece of \eqref{czDecomp} we wish to control, we let $b_{n}$ be the first mean-zero term that appears (ordered from left to right) and then apply the $n$-th estimate in the list \eqref{assumptionL}, ordered from top to bottom. Then applying the CZ properties as above yields the desired estimates. \end{proof}
		
		\noindent The construction of the collection $\mathcal{S}$ now proceeds as in \cite{roughSparse}, with minor changes to account for the addition of more functions. For this reason we sketch the proof here, and send the reader to \cite{roughSparse} for the finer details. 
		
		We will construct the required sparse collection $\mathcal{S}$ iteratively by decomposing exceptional sets of the form \begin{equation}\label{excSet}E_{Q} := \left\{ x \in 3Q : \max_{i=1,...,m+1} \frac{M_{p_i}(f_i \textbf{1}_{3Q})(x)}{ (f_i)_{p_i, 3Q} } \geq C_d \right\},\end{equation} where $C_d$ is some constant depending on the dimension. Notice that if $C_d$ is chosen large enough, then by the maximal theorem we can assume \begin{equation}\label{sparseEst}|E_Q| \leq 2^{-cd}|Q|\end{equation} for some uniform $c > 0$. 
		
		We begin the argument by fixing $f_i \in L^{p_{j}}(\R^{d})$ with compact support, $i = 1,...,m+1$. We may assume we have chosen our dyadic lattice $\mathcal{D}$ so that there is $Q_0 \in \mathcal{D}$ with supp($f_1$)$\subset Q_0$ and supp($f_i$)$\subset 3Q_0$ for $i=2,...,m+1$, such that $s_{Q_0}$ is bigger than the largest scale occurring in the truncation of $\Lambda$. We let $S_0 = \{Q_0 \}$ and $E_0 = 3Q_0$, and then define (using definition \eqref{excSet}) $$E_1 := E_{Q_0}, \ \ \ \  S_1 := \text{ maximal cubes $L \in \mathcal{D}$ such that $9L\subset E_1$}.$$ It is easy to verify that $\mathcal{P}_1(Q_0) := S_1$ is a stopping collection with top $Q_0$, such that $|Q_0 \backslash E_1 | \geq (1- 2^{-cd})|Q_0|.$ By maximality and the definition of the exceptional set $E_{Q_0}$ we have $$ \|f_j\|_{\mathcal{Y}_{p_j} (\mathcal{P}_1(Q_0)) } \leq C_{d}(f_i)_{p_i, 3Q_0}, \ \ \ i = 1,...,m+1,$$ so applying Lemma \ref{sparseInit} yields 
		
		\begin{align*}|\Lambda(f_1,f_2,...,f_{m+1})| &= |\Lambda^{s_{Q_0}}(f_1\textbf{1}_{Q_0}, f_2 \textbf{1}_{3Q_0}, ..., f_{m+1}\textbf{1}_{3Q_0})| \\ &\leq \ C_{d}|Q_0|\prod_{i=1}^{m+1}(f_i)_{p_i, 3Q_0} \\ & \ \ \ \ \ \ \ \ \ \ \  + \sum_{\substack{L \in \mathcal{P}_1(Q_0) \\ L \subset Q_0 }}|\Lambda^{s_{L}}(f_1\textbf{1}_{L}, f_2 \textbf{1}_{3L},..., f_{m+1}\textbf{1}_{3L})|.\end{align*} This is the base case that sets up the recursive argument. For each $L \in \mathcal{P}_1(Q_0)$ we apply the above argument with $L$ in place of $Q_0$, defining $$E_2 = \bigcup_{L \in \mathcal{P}_1 (Q_0)} E_L$$ and $$S_2 := \text{ maximal cubes $R \in \mathcal{D}$ such that $9R \subset E_2$}.$$ For each $L \in \mathcal{P}_1(Q_0)$ we also define $$\mathcal{P}_2(L) = \{R \in S_2 : R \subset 3L\}, $$ which can be shown to be a stopping collection with top $L$ (the argument is the same as in \cite{roughSparse}). Now we apply Lemma \ref{sparseInit} to estimate each piece of the sum $$ \sum_{\substack{L \in \mathcal{P}_1(Q_0) \\ L \subset Q_0 }}|\Lambda^{s_{L}}(h_1\textbf{1}_{L}, h_2 \textbf{1}_{3L},..., h_{m+1}\textbf{1}_{3L})|,$$ using the stopping collection $\mathcal{P}_2(L)$ in the application to the term $$|\Lambda^{s_{L}}(f_1\textbf{1}_{L}, f_2 \textbf{1}_{3L},..., f_{m+1}\textbf{1}_{3L})|.$$ We then repeat the process just described at the next level. 
		
		Since we are working with finitely many scales the process eventually terminates with the desired sparse form bound. The sparse collection $\mathcal{S}$ consists of all cubes chosen in the various stopping collections constructed along the way. Condition \eqref{sparseEst} at each level guarantees the sparsity of these cubes. 
		
		\subsection{Multiplication Operators} Let $A_{\phi}: L^{r_1}\times...\times L^{r_m}\rightarrow L^{\alpha}$ be the multiplication operator described in Remark 2.1, with $\phi \in L^{\infty}$ and $\frac{1}{r_1}+...+\frac{1}{r_m} = \frac{1}{\alpha}.$ We let $$A(f_1,...f_{m+1}) = \langle A_{\phi}(f_1,...,f_{m}), f_{m+1} \rangle.$$ Using a simpler version of the stopping-time argument from Section 6.1, it is easy to see that $$|A(f_1,...,f_m, f_{m+1})| \lesssim_{\phi} \sup_{\mathcal{S}}\text{PSF}^{(1,1,...,1)}_{\mathcal{S}}(f_1,...,f_{m+1})$$ when $f_{i}\in L^{r_i}\cap L^{\infty}$ with compact support. Let $\mathcal{P}$ be any stopping collection with top $Q$ and $b = \sum_{L \in \mathcal{P}}b_L$ with $b_L$ supported on $L$, and observe that if \begin{align*}A_{\mathcal{P}}(b,g_2,...,g_{m+1})&:= A(b\textbf{1}_{3Q}, g_2\textbf{1}_{3Q},...,g_{m+1}\textbf{1}_{3Q}) \\ &  \ \ \ \ \ \ \ \ \ - \sum_{\substack{R \in \mathcal{P} \\ R \subset Q }} A(b\textbf{1}_{3R}, g_2 \textbf{1}_{3R},...,g_{m+1}\textbf{1}_{3R}) \end{align*} then in fact $$ A_{\mathcal{P}}(b,g_2,...,g_{m+1}) = 0. $$  Similarly $A_{\mathcal{P}}(g_1, b, g_{3},...,g_{m+1}),...,A_{\mathcal{P}}(g_1,...,g_{m}, b)$ all vanish. Hence we can repeat the argument given in the proof of Theorem \ref{abstractThm}, avoiding most of the technical complications. In particular, if $f_i = g_{i} + b_{i}$ with $g_i$ the `good' terms from the Calder\'{o}n-Zygmund decomposition relative to $\mathcal{P}$, then we can argue as in the proof of Lemma \ref{sparseInit} to show $$|A_{\mathcal{P}}(g_1,...,g_{m+1})|\lesssim \|\phi\|_{L^{\infty}}|Q|\|f_1\|_{\mathcal{Y}_{1}} \|f_2\|_{\mathcal{Y}_{1}}... \|f_{m+1}\|_{\mathcal{Y}_{1}}.$$ We just saw that all other terms of $A_{\mathcal{P}}(g_1 + b_1,...,g_{m+1}+b_{m+1})$ vanish, so we can easily run the stopping-time argument given in the last section to prove the claimed $\text{PSF}_{\mathcal{S}}^{(1,...,1)}$ bound.

		\section*{Appendix: Adjoint Forms} 
		
		We prove Proposition \ref{adjointDecomp}. The argument is almost identical to the proof of the linear variant from \cite{roughSparse}, so we only provide a sketch. Below we will call two dyadic cubes $L,R$ \textit{neighbors} and write $L \sim R$ if $7L\cap 7R \neq \emptyset$ and $|s_L - s_R|<8$. By separation property (i), if $L,R \in \mathcal{P}$ are distinct cubes with $7L\cap 7R \neq \emptyset$, $L \sim R$. 
		
		Let $\mathcal{P}$ be a stopping collection with top $Q$. Let $b = \sum_{L \in \mathcal{P}} b_L$ as in the last sections. We want to show that $$\Lambda_{\mathcal{P}}(g_1, b, g_3,...,g_{m+1}), ..., \Lambda_{\mathcal{P}}(g_1,...,g_{m}, b)$$ can be decomposed in the same way as $\Lambda_{\mathcal{P}}(b,g_2,...,g_{m+1})$, up to a controllable error term. We first analyze $\Lambda_{\mathcal{P}}(g_1,...,g_m, b).$ Below we assume $g_1$ is supported in $Q$ and $g_i$ is supported in $3Q$ for $i \geq 2.$ 
		
		As in the appendix in \cite{roughSparse}, split $b = b^{in} + b^{out}$ with $$b^{in} = \sum_{\substack{ L \in \mathcal{P} \\ 3L \cap 2Q \neq \emptyset }} b_L.$$ Fix any $R \subset Q$. Using the support of the kernel, it is easy to see that if $s < s_R$ then $$\int_{\R^{(m+1)d}}K_{s}(x_1,...,x_{m})g_1(x_1)\textbf{1}_{R}(x_1)g_2(x_2)...g_m(x_m)b^{out}(x_{m+1}) dx_1...dx_{m+1}$$ is identically zero. This is because dist(supp $b^{out}$, $R$) $\geq l(R)/2,$ since $b^{out}$ is supported on $L$ with $3L \cap 2Q = \emptyset$, but $R \subset Q$. Therefore
		
		\begin{align*}\Lambda_{\mathcal{P}}(g_1,...,g_m, &b^{out}) = \int_{\R^{(m+1)d}}\sum_{i,j}K^{ij}_{s_Q}(\vec{x})g_1(x_1)...g_m(x_{m})b^{out}(x_{m+1}) d\vec{x}\\ &- \sum_{\substack{R \in \mathcal{P} \\ R\subset Q}}\int_{\R^{(m+1)d}}\sum_{i,j}K^{ij}_{s_R}(\vec{x})g_1\textbf{1}_{R}(x_1)...g_{m}(x_m)b^{out}(x_{m+1}) d\vec{x} \end{align*}
		
		\noindent There is only one scale in the kernel in each integral, so we can estimate each term as in Lemma \ref{gbSS1} and sum over $R$ to see that $$|\Lambda_{\mathcal{P}}(g_1,...,g_{m}, b^{out})| \lesssim [K]_{p'} |Q| \|b\|_{\mathcal{X}_{1}}\|g_1\|_{\mathcal{Y}_{p}}...\|g_{m}\|_{\mathcal{Y}_{p}}.$$ 
		
		\noindent We now have to analyze $\Lambda_{\mathcal{P}}(g_1,...,g_m,b^{in}).$ As in the appendix of \cite{roughSparse}, it is enough to show that \begin{align}\label{inEst}\Lambda_{\mathcal{P}}(g_1,...,g_m, b^{in}) = &\left( \Lambda^{s_Q}(g_1,...,g_m, b^{in}) - \sum_{\substack{L \in \mathcal{P} \\ 3L \cap 2Q \neq \emptyset }} \Lambda^{s_L}(g_1,...,g_m,b_L) \right) \\ \nonumber & \ \ \ \ \ \ \ + \phi(g_1,...,g_m,b),\end{align} where $\phi$ is some error term satisfying the same estimate as $|\Lambda_{\mathcal{P}}(g_1,...,g_m, b^{out})|$. Then we can decompose the term in parenthesis as in Section 3, and use symmetry of the kernel to prove the desired estimates for $\Lambda_{\mathcal{P}}(g_1,...,g_m,b)$.
		
		The proof of \eqref{inEst} is almost the same as the linear case in \cite{roughSparse}. The main observation is that \begin{equation}\label{inDecomp}\sum_{\substack{R \in \mathcal{P} \\ R \subset Q }}\Lambda^{s_R}(g_1\textbf{1}_{R}, g_2,..., g_m, b^{in}) = \sum_{\substack{R \in \mathcal{P} \\ R \subset Q }}\sum_{\substack{ L \in \mathcal{P} \\ 3L \cap 3R \neq \emptyset \\ 3L \cap 2Q \neq \emptyset  }} \Lambda^{s_R}(g_1\textbf{1}_{R}, g_2,..., g_m, b_L),\end{equation} 
		
		\noindent since $\Lambda_R(g_1\textbf{1}_{R},g_2,..., g_m, b_L) = 0$ if $3L \cap 3R = \emptyset$. In particular this implies that $L \sim R$, so if $R$ is fixed then the number of terms in the second sum in \eqref{inDecomp} is bounded by a universal dimensional constant. It is easy to show that \begin{align*}|\Lambda^{s_R}(g_1\textbf{1}_{R}, g_2,..., g_m, b_L)& - \Lambda^{s_L}(g_1\textbf{1}_{R},g_2,..., g_m, b_L)| \\ &\lesssim |L|[K]_{p'}\|g_1\|_{\mathcal{Y}_{p}}...\|g_m\|_{\mathcal{Y}_{p}}\|b\|_{\mathcal{X}_{1}},\end{align*} since $|s_L - s_R| < 8$ (apply single-scale estimates, noting that the number of such estimates will be bounded by a dimensional constant). With the help of the separation properties we then can replace each $\Lambda^{s_R}$ in \eqref{inDecomp} by $\Lambda^{s_L}$, up to an admissible error term $\phi$. Now repeat the rest of the argument from \cite{roughSparse}, with trivial changes to account for the addition of more functions, to show that the remaining terms are of the form \eqref{inEst}. 
		
		Also observe that there was nothing special about the choice of the position of $b$ in the above argument, so the same reasoning applies to $\Lambda_{\mathcal{P}}(g_1, b, g_3,...,g_{m+1}), \Lambda_{\mathcal{P}}(g_1, g_2, b, g_4,...,g_{m+1})$, etc. This proves Proposition \ref{adjointDecomp}, up to the trivial steps we have omitted.

\Addresses

\end{document}